\documentclass[11pt,leqno]{amsart}
\usepackage[dvipsnames]{xcolor}
\usepackage{amsmath}
\usepackage{amsfonts}
\usepackage{amssymb}
\usepackage{graphicx}
\usepackage{hyperref} 
\usepackage{mathrsfs} 
\usepackage{relsize} 

\theoremstyle{plain}
\newtheorem{theorem}{Theorem}

\newtheorem{corollary}[theorem]{Corollary}
\newtheorem{definition}[theorem]{Definition}
\newtheorem{lemma}[theorem]{Lemma}
\newtheorem*{problem*}{Problem}
\newtheorem{problem}{Problem}
\newtheorem{proposition}[theorem]{Proposition}

\theoremstyle{definition}
\newtheorem{example}[theorem]{Example}

\newtheorem{notation}[theorem]{Notation}
\newtheorem{remark}[theorem]{Remark}

\numberwithin{equation}{section}
\numberwithin{theorem}{section}

\newcommand{\area}{\operatorname{area}}
\newcommand{\vol}{\operatorname{vol}}

\newcommand{\inte}{\mathrm{int}}
\newcommand{\pM}{\partial M}

\renewcommand{\div}{\operatorname{div}}

\newcommand{\Hess}{\operatorname{Hess}}
\newcommand{\DeltaN}{\Delta\!^{N}}
\newcommand{\DeltaM}{\Delta\!^{M}}

\newcommand{\rr}{\mathbb{R}}
\renewcommand{\ss}{\mathbb{S}}

\newcommand{\nn}{\mathbb{N}}

\newcommand{\bb}{\mathbb{B}}
\renewcommand{\gg}{\mathbb{G}}
\newcommand{\hh}{\mathbb{H}}
\newcommand{\mm}{\mathbb{M}}

\newcommand{\sect}{\operatorname{Sect}}
\newcommand{\ric}{\operatorname{Ric}}

\newcommand{\tr}{\operatorname{trace}}

\newcommand{\dv}{\mathrm{dv}}
\newcommand{\da}{\mathrm{da}}
\newcommand{\dr}{\mathrm{d}r}
\newcommand{\ds}{\mathrm{d}s}
\newcommand{\dz}{\mathrm{d}z}

\def\XXint#1#2#3{{\setbox0=\hbox{$#1{#2#3}{\int}$}
     \vcenter{\hbox{$#2#3$}}\kern-.5\wd0}}

\newcommand{\dist}{\mathrm{dist}}

\renewcommand{\a}{\alpha}
\renewcommand{\b}{\beta}

\newcommand{\g}{\gamma}
\renewcommand{\l}{\lambda}

\newcommand{\s}{\sigma}
\newcommand{\vp}{\varphi}

\renewcommand{\O}{\Omega}

\newcommand{\CA}{\mathcal{A}}

\newcommand{\CE}{\mathcal{E}}

\newcommand{\CN}{\mathcal{N}}

\newcommand{\CL}{\mathcal{L}}

\newcommand{\D}{\mathscr{D}}

\renewcommand{\S}{\mathscr{S}}

\begin{document}

\begin{abstract}
This paper deals with symmetry phenomena for  solutions of the Dirichlet problem involving semilinear PDEs on Riemannian domains.  We shall present a rather general framework where the symmetry problem can be formulated and provide some evidence that this framework is completely natural  by pointing out some results for stable solutions. The case of manifolds with density, and corresponding weighted Laplacians, is inserted in the picture from the very beginning.
\end{abstract}


\title[Symmetry of stable solutions of semilinear PDEs]{Symmetry of solutions of semilinear PDEs on Riemannian domains}
\author{Andrea Bisterzo}
\address{Universit\`a degli Studi di Milano-Bicocca\\ Dipartimento di Matematica e Applicazioni \\ Via Cozzi 55, 20126 Milano - ITALY}
\email{a.bisterzo@campus.unimib.it}
\author{Stefano Pigola}
\address{Universit\`a degli Studi di Milano-Bicocca\\ Dipartimento di Matematica e Applicazioni \\ Via Cozzi 55, 20126 Milano - ITALY}
\email{stefano.pigola@unimib.it}
\date{\today}
\maketitle


\section{Introduction}
This paper deals with symmetry phenomena for  solutions of the Dirichlet problem involving semilinear PDEs on Riemannian domains.  We shall present a rather general framework where the symmetry problem can be formulated and provide some evidence that this framework is completely natural  by pointing out some results for stable solutions. The case of manifolds with density, and corresponding weighted Laplacians, is inserted in the picture from the very beginning. The investigations of the present paper all arise from the {\it elementary properties of stable solutions} in Euclidean domains as they are presented by L. Dupaigne in \cite[Section 1.3]{Du} and show how much geometry was (more or less implicitly) contained there.

\subsection{Basic notation}

Throughout this paper, $(M,g)$ will always denote a connected Riemannian manifold of dimension $\dim M = m$. The symbols $\sect$ and $\ric$ are reserved to its  sectional and  Ricci curvatures. We set $\dist(x,y)$ for the intrinsic distance of $M$. The corresponding open metric ball centered at $o \in M$ and of radius $R>0$ is $B^{M}_{R}(o) = \{ x \in M : \dist(x,o) < R\}$. When there is no danger of confusion, the overscript $M$ is omitted in the notation and we simply write $B_{R}(o)$. Moreover, in the special case where $M=\rr^{n}$ is equipped with its standard flat metric $g^{E}$ we set $\bb_{R} = B_{R}(0)$.

A  class of Riemannian manifolds of special interest is that of {\it model manifolds}. Let $\s:[0,R) \to \rr_{\geq 0}$, $0<R \leq +\infty$, be a smooth function that is positive in $(0,R)$ and satisfying
\begin{itemize}
 \item $\s^{(2k)}(0) = 0$ for all $k \in \nn$;
 \item $\s'(0) = 1$.
\end{itemize}
Then, in polar coordinates around $0$, we can define a smooth Riemannian metric on $(0,R)\times \ss^{m-1}$ by setting
\[
g = dr \otimes dr + \s^{2}(r) g^{\ss^{m-1}},
\]
where $g^{\ss^{m-1}}$ is the standard metric on the unit sphere $\ss^{m-1}\subset \rr^{m}$. The corresponding Riemannian manifold $\mm^{m}(\s)=(\mathbb{B}_R,g)$, obtained by identifying all the points of the form $(0,\theta)$ with $0$ and extending (smoothly) the metric in $0$, will be called an $m$-dimensional {\it model manifold with warping function $\s$}. Clearly, $\mm(\s)$ is complete if and only if $R=+\infty$ and, in any case, the $r$-coordinate represents the distance from the {\it pole} $o=0\in\rr^{m}$. Thus, $B^{\mm(\s)}_{T}(o) = \{ x \in \bb_{R}: r(x) < T\}$. For more details on the construction of warped product manifolds and model manifolds we suggest \cite{Pe}.

\begin{example}
 The standard spaceforms $\rr^{m}$, $\ss^{m} \setminus \{ pt.\}$ and $\hh^{m}$ are model manifolds with the choice, respectively, $\s(r) = r$, $\s(r) = \sin(r)$, $\s(r) = \sinh(r)$. 
\end{example}

Now, let the Riemannian manifold $(M,g)$ be endowed  with the absolutely continuous measure $\dv_{\Psi} = e^{-\Psi}\dv$ where $\dv$ is the Riemannian measure and $\Psi:M \to \rr$ is a selected smooth function. Usually, the triple
\[
M_{\Psi} = (M,g,\dv_{\Psi})
\]
is called a  {\it weighted manifold} or a {\it manifold with density} or a {\it smooth metric measure space}.

On the weighted manifold $M_{\Psi}$ we have a natural linear elliptic differential operator. It is the {\it weighted Laplacian}, also called {\it $\Psi$-Laplacian}, which is defined by the formula
\[
\Delta_{\Psi} u = e^{\Psi} \div(e^{-\Psi} \nabla u) = \Delta u - g(\nabla \Psi,\nabla u).
\]
Here,
\[
\Delta u = \tr \Hess(u) = \div (\nabla u)
\]
stands for the {\it Laplace-Beltrami} operator of $(M,g)$. We stress that we are using the sign convention according to which, in case $M = \rr$, $\Delta = + d^{2}/dx^{2}$. In other terms, $\Delta$ is a negative definite operator in the spectral sense. Note also that when $\Psi \equiv const$ then $\Delta_{\Psi} = \Delta$.

Very often, one sets
\[
\div_{\Psi} X = e^{\Psi} \div(e^{-\Psi} X)
\]
so that the $\Psi$-Laplacian takes the suggestive form
\[
\Delta_{\Psi} u = \div_{\Psi}(\nabla u).
\]
Clearly, we have the validity of the $\Psi$-divergence theorem on $M_{\Psi}$: given a compact domain $\O $ with smooth boundary and a vector field $X$, it holds
\[
\int_{\O } \div_{\Psi} X\ \dv_{\Psi} = \int_{\partial \O } g(X,\vec \nu)\ \da_{\Psi}
\]
where $\vec \nu$ is the exterior unit normal to $\partial \O $, $\da_{\Psi} = e^{-\Psi} \da$ and $\da$ is the $(m-1)$-dimensional Hausdorff measure of $\partial \O $. As a simple consequence, the operator $\Delta_{\Psi}$ is symmetric on $L^{2}(M,\dv_{\Psi})$.\smallskip

The geometric analysis on the weighted manifold $M_{\Psi}$ is influenced by the bounds of its family of Bakry-Emery Ricci tensors. In view of our purposes we limit ourselves to introduce the $\infty$-dimensional Ricci Tensor
\[
\ric_{\Psi} = \ric + \Hess(\Psi).
\]

\begin{example}
 The Gaussian space
\[
\gg^{m} = \left(\rr^{m},g^{\rr^{m}}, e^{-\frac{|x|^{2}}{2}}dx \right)
\]
is an example of great interest in metric and differential geometry, probability, harmonic and geometric analysis. Its weighted Laplacian $\Delta_{\Psi} u = \Delta u - \langle \nabla u , x\rangle$ is the {\it Ornstein-Uhlenbeck operator}. Obviously the Gaussian space is a weighted model manifold
\[
\gg^{m} = \mm^{m}(\s)_{\Psi}
\]
with warping function $\s(r) = r$ and symmetric weight $\Psi(x) = r^{2}(x)/2$. A direct computation shows  that $\ric_{\Psi} \equiv 1$.
\end{example}

\subsection{Symmetry under stability}

We are going to address the following  classical
\begin{problem}\label{problem}
Let $\O $ a (possibly non-compact) domain in the weighted Riemannian manifold $M_{\Psi}$ and assume that $\O$ has smooth boundary components $\partial \O  = (\partial \O)_{1}\cup \cdots \cup (\partial \O)_{n}$. Let us given regular solution of the semilinear boundary value problem
\begin{equation}\label{DP}
\begin{cases}
\Delta_{\Psi} u = f(u) & \text{in }\O  \\
u = \phi_{j} & \text{on }(\partial \O)_{j}
\end{cases}
\end{equation}
for some sufficiently regular nonlinearity $f(t)$. Assume that the domain, the  differential operator and the boundary data display a certain (and same) symmetry. To what extent the solution inherits this symmetry? 
\end{problem}

We stress that our solutions will be always assumed to be very regular (say, at least $C^2$). The case of weakly regular solutions introduces nontrivial difficulties and require further assumptions, as one can see from the very recent \cite{DF} by Dupaigne and Farina.

In the Euclidean space $M=\rr^{n}$, the celebrated theorem by B. Gidas, W.M. Ni and L. Nirenberg, \cite{GNN}, later extended to spherical and hyperbolic spaceforms in \cite{KP2}, states that if $\O  = \bb$ is the (unit) ball of $\rr^{n}$,  $\Delta_{\Psi} = \Delta$ is the Euclidean Laplacian and $\phi \equiv 0$, then any solution  $u > 0$ of \eqref{DP} is rotationally symmetric (and decreasing). The proof makes use of the {\it moving plane method} and, therefore, requires a lot of homogeneity of the underlying space in order to perform reflections in every direction.  It is well known that the positivity of the solution is vital as shown by the (non-symmetric) eigenfunctions relative to higher Dirichlet eigenvalues of the ball. Moreover, the ball itself cannot, in general, be replaced by a non-convex domain, like an annulus, as the seminal example by H. Brezis and L. Nirenberg shows, \cite[p. 453]{BN}.

However, as we are going to see in a quite general geometric setting and as it is proved by N.D. Alikakos and P.W. Bates, \cite{AB}, in the Euclidean space, both these assumptions become redundant as soon as it is assumed that the solution $u$ is ``stable''.\smallskip

In fact, in this paper we shall only focus the case of {\it stable solutions} of \eqref{DP}, where the nonlinearity $f(t)$ is at least $C^{1}$. Stability is a second order condition defined in terms of the first Dirichlet eigenvalue of the linearized (Schr\"odinger) operator and it is always satisfied if the solution is energy minimizer. More precisely, assume for simplicity that $\O $ is compact. Let $F(t)$ be a primitive of the $C^{1}$ function $f(t)$ and consider the energy functional
\[
\CE[v] = \int_{\O } \left( \frac{1}{2}|\nabla v |^{2} + F(v) \right) \dv_{\Psi}
\]
on the space
\[
\S = \{ v \in C^{2}(\overline\O ) : v|_{(\partial \O)_{j} } = \phi_{j} \}.
\]
For any $\vp \in C^{\infty}_{c}(\O )$ and $t \in \rr$ it holds $u_{t} = u + t \vp \in \S$. If $u$ is a classical solution of the problem, then (integrating by parts) $u$ is a weak solution of the PDE and, therefore
\[
\left.\frac{d}{dt}\right\vert_{t=0} \CE[u_{t}] =  \int_{\O } g(\nabla u, \nabla \vp) \, \dv_{\Psi}  + \int_{\O }f(u)  \vp \, \dv_{\Psi}= 0.
\]

\begin{definition}[Stable and strongly stable solutions]
Say that the solution $u$ is stable if
\[
0 \leq \left.\frac{d^{2}}{dt^{2}}\right\vert_{t=0} \CE[u_{t}] = \int_{\O }\left(|\nabla \vp|^{2} +f'(u) \vp^{2}\right)\dv_{\Psi}
\]
i.e. the stability operator $\CL= \Delta_{\Psi} - f'(u)$ has nonnegative Dirichlet spectrum:
\[
\l_{1}^{-\CL}(\O) := \inf_{\vp \in C^{\infty}_{c}(\O ), \, \vp \not\equiv 0} \frac{\int_{\O }(|\nabla \vp|^{2} +f'(u) \vp^{2})\dv_{\Psi}}{ \int_{\O } \vp^{2}\dv_{\Psi}} \geq 0.
\]
The solution $u$ is said to be \textit{strongly stable} if $\l_1^{-\CL}(\O)>0$. 
\end{definition}

\subsection{Organization of the paper}
Clearly, in order to carry out an investigation around Problem \ref{problem}, we need first to clarify what ``symmetric'' means for a Riemannian domain and, hence, for a solution of \eqref{DP} on this domain. We choose to define the symmetry of a domain in terms of the existence of a  foliation by special  hypersurfaces and the corresponding symmetry of functions as the condition that the function is constant on each leaf of the foliation. Equivalently, the function agrees with its averages on the (compact) leaves of the foliation. This is explained in Sections \ref{section-Symmetric} and \ref{section-symmetricfunctions}. \smallskip

Sometimes, and these are the lucky cases, symmetry properties of generic solutions boil down to uniqueness issues for the relevant class of PDEs. In Section  \ref{section-maximum-unique-symmetry} we review (slightly extended versions of) both the classical maximum principle for Schr\"odinger operators and the uniqueness property of stable solutions. As a consequence of the maximum principle and the fact that the average operator commutes with the differential operator, we observe how, in this general geometric framework, symmetry over compact domains occurs for affine $f(t)$.\smallskip

In Section \ref{section-geometicDupaigne} we point out that symmetry of stable solutions appears as soon as the domain supports enough Killing vector fields tangential to the leaves of its foliation. This translates the fact that the domain is {\it homogeneous} in the precise sense of {\it co-homogenity one actions of Lie subgroups of isometries}. This simple result encloses in a single view a lot of concrete cases that, at first glance, could appear of different nature, such as balls in model manifolds, annuli in warped products of a real interval with a homoegeneous manifold, tubes around Clifford tori in the $n$-sphere and many others. \smallskip

\noindent In Section \ref {SymmetryWithoutKillingFields}, in order to test how much the existence of infinitesimal symmetries influence the problem, we consider the case of a possibly non-compact warped product that, in general, supports no Killing fields at all. Using potential theoretic tools, we are still able to prove a quite general symmetry result for (strongly stable) solutions provided the nonlinearity is concave and somewhat compatible with the geometry. The general result applies e.g. to slabs (the region enclosed between two parallel hyperplanes) in the Gaussian space.\bigskip

\noindent{\bf Acknowledgments. }
The authors would like to thank Giona Veronelli for his suggestions related to the proof of Lemma \ref{Lem:L1ContinuityEnlarging} and Alberto Farina for explanations about the content of \cite{FMV} and for some interesting discussions concerning maximum principles.


\section{Symmetric domains}\label{section-Symmetric}

As we have already mentioned in the Introduction, the first aspect we need to clarify is what does ``symmetric'' mean in the setting of Riemannian manifolds. At first glance, ``radial symmetry'' could appear the most natural notion. However, the recent and very active area of research on the geometry of overdetermined problems of various nature, strongly suggests that the appropriate notion is that of an {\it isoparametric domain}; see especially the seminal paper \cite{Sh} by V. Shklover, the papers \cite{savo2016heat, savo2018geometric} by A. Savo and the very recent \cite{savoprovenzano2021} by L. Provenzano and A. Savo.

Isoparametric hypersurfaces in space-forms have a long history that goes back to the first half of the nineteen century and the modern viewpoint on this theory can be attributed to E. Cartan, \cite{Ca}. For a gentle introduction on the subject, with plenty of examples and special emphasis on the classification problem in different ambient spaces, we refer the reader to the lecture notes \cite{Va} by M. Dominguez-Vazquez and the references therein.\smallskip

\subsection{Isoparametric domains and tubes}
We recall that a \textit{singular Riemannian foliation} of the Riemannian manifold $(M,g)$ is a foliation $M=\cup_t \Sigma_t$ by smooth, embdedded submanifolds such that:
\begin{itemize}
\item every geodesic that is perpendicular to one leaf remains perpendicular to every leaf that it intersects;
\item there exists a family of smooth vector fields (integrable distribution) $\D = \{ X_1,...,X_k \}$  on $M$ spanning pointwise every tangent space to all the leaves.
\end{itemize} 

\begin{definition}[Isoparametric domain]\label{Def:IsoparametricDomain}
An isoparametric domain $\bar \O\subseteq M$ is a domain of $M$ endowed by a singular Riemannian foliation $\bar \O=\cup_t \Sigma_t $ whose regular leaves (i.e. of maximal dimension) are connected parallel hypersurfaces with constant mean curvature.
\end{definition}
Here, as usual, we call $\Sigma_{1}, \Sigma_{2}$ parallel if, for every $x_{1}\in \Sigma_{1}$ and $x_{2}\in \Sigma_{2}$,
\[
\dist(x_{1},\Sigma_{2}) = \dist(\Sigma_{1},x_{2}).
\]

Constant mean curvature hypersurfaces that influence the geometry of nearby parallel hypersurfaces, i.e. such that sufficiently close parallel hypersurfaces have constant mean curvature, are called \textit{isoparametric hypersurfaces}. Thus, an isoparametric domain is nothing but a domain with a singular Riemannian foliation whose regular leaves are isoparametric hypersurfaces.
 
Smooth isoparametric hypersurfaces in Riemannian manifolds arise as regular level sets of \textit{isoparametric functions}, i.e. smooth functions $f$ whose norm of the gradient and whose Laplacian can be expressed in terms of the function itself:
\[
|\nabla f | = \alpha(f) \quad \text{and} \quad \Delta f= \beta(f).
\]
 These two properties imply respectively that level sets are parallel and with constant mean curvature.
 
If every leaf of the foliation of an isoparametric domain is regular (and thus orientable), then the leaves can be realized as the level sets of the signed distance function $\dist(\bullet,P)$ from any fixed leaf $P$. Similarly, if the domain at hand has at least one focal variety $P$ (for instance, if the domain is compact), then the leaves of the foliation are level sets of the positive distance function from $P$. In both these cases, the submanifold $P$ is called the \textit{soul of the isoparametric domain}. This characterization allows one to name the leaves as \textit{equidistants}.

\begin{remark}\label{Remark:FocalVarieties} Observe that if the manifold $M$ is complete, then the focal varieties are smooth minimal submanifold of $M$ and are at most two (\cite{Wa}).
\end{remark}

\subsection{Homogeneous domains}
The isoparametric condition provides a very handy model of symmetric domains. However, as we shall see, sometimes the needed notion of symmetry is much  stronger. 

\begin{definition}\label{Def:HomogeneousDomain}
A homogeneous domain $\bar \O\subseteq M$ of a complete Riemannian manifold $(M,g)$ is an isoparametric domain whose regular leaves are orbits of the action of a closed subgroup $G\subset \textnormal{Iso}_{0}(M)$, the identity component of the group $\mathrm{Iso}(M)$ of all isometries of $M$.
\end{definition}

Thus, a  domain is homogeneous if the regular leaves of the singular Riemannian foliation are homogeneous hypersurfaces with respect to the same group $G$ of isometries of the ambient space.

A straightforward consequence of the fact that $G$ acts transitively on each leaf is that the principal curvatures of the leaves are constant. Moreover, note explicitly that if $\dim M =m$, since each regular leaf is homogeneous and can be written as $\Sigma_t=G/H_p$ for $H_p\subset G$ isotropy subgroup of $G$ at $p\in \Sigma_t$, then $\dim G=k\geq  m-1$. \smallskip

From the perspective of the present paper, the most important property enjoyed by homogenenous domains is that the leaves display a lot of (and in fact same) isometric symmetries. These symmetries are encoded in the notion of a Killing vector field that we are going to recall.

A smooth vector field $X$ on $M$ is said to be {\it Killing} if, for every vector fields $Y,Z$,
\[
(L_{X}g)(Y,Z) = g(\nabla_{Y}X,Z) + g(\nabla_{Z}X,Y) =0.
\]
Equivalently, the flow  $\phi(x,t)$ of $X$ is a local $1$-parameter group of isometries:
\[
\phi_{t}^{\ast} g = g.
\]
Note that, by the very definition, any Killing vector field $X$ satisfies
\[
\div X =0.
\]
Note also that if $X$ is a Killing vector field on $(M,g)$, which is pointwise tangential to an embedded submanifold $P$, then $X|_{P}$ is a Killing vector field of $P$.\smallskip

Now, let $\bar \O$ be a homogeneous domain with group $G$ and whose regular leaves are homogeneous hypersurfaces $\Sigma_{t}$ and recall from Remark \ref{Remark:FocalVarieties} that $\bar{\O}$ has at most two focal varieties $P_1$ and $P_2$. Consider the Riemannian submersion given by the projection
\[
\begin{array}{ccc}
 \pi:  \bar \O\setminus (P_1 \cup P_2) & \longrightarrow & \rr \\
        \Sigma_{t} & \longmapsto & \Sigma_{t}/G=point
\end{array}
\]
and note that
\begin{align}\label{Eq:VerticalSpace}
\mathcal{V}_p=T_p\Sigma_t\ \ \ \ \forall p \in \Sigma_t
\end{align}
where $\mathcal{V}_p=\textnormal{Ker}(d_{p}\pi)$ is the vertical space at $p$. For any $p\in \Sigma_t$ the space $\mathcal{V}_p$ is spanned by the set $\frak{K}(\bar \O)$ of all Killing vector fields of $\bar \O$ evaluated at $p$. These, in turn, identify with the elements of the Lie algebra $\frak{g}$ of $G$ via the map
\begin{align*}
\begin{array}{ll}
\frak{g}&\longrightarrow \frak{K}(\bar \O)\\
\frak{X}&\longmapsto X
\end{array}
\end{align*}
where
\[
X:p\mapsto \left.\frac{d}{dt}\right\vert_{t=0}\Big(\exp(t\frak{X})(p)\Big).
\]
Thus, letting $m-1 \leq k=\dim G \leq m(m-1)/2$, we can select a distribution of linearly independent Killing vector fields
\[
\D = \{X_{1},\cdots,X_{k}\} \subseteq \frak{K}(\bar \O)
\]
whose integral manifolds are the hypersurfaces $\Sigma_{t}$. For further information on the topic we suggest \cite{Pe}.

\subsection{Examples}
It is time to present a brief list of concrete examples of isoparametric and homogenenous domains.

\begin{example}[Balls in model manifolds]\label{ex: balls-homog}
Let $\mathbb{M}^n_\sigma = [0,R)\times_\sigma \mathbb{S}^{n-1}$ be a model manifold, where $R \in (0,+\infty]$. Then, geodesic balls centred at the pole are homogeneous domains with the homogeneous foliation provided by the geodesic spheres concentric to the pole. The corresponding group is $G=\mathbf{SO}(n)$.
\end{example}

\begin{example}[Annuli in warped products]\label{Ex:WarpedProduct}
Take a warped product manifold $M = I \times_{\s} N$ where $(N,g^{N})$ is an $(m-1)$-dimensional Riemannian manifold without boundary, $I\subset \rr$ is a real open interval and $\s(t) > 0$ is a smooth function on $I$. Explicitly, the Riemannian metric $g$ of $M$ is given by
\[
g = dt\otimes dt + \s^{2}(t) g^{N}.
\]
Take a domain either of the form $\bar \Omega = [a,b] \times N$ or $\bar \Omega = [a,+\infty) \times N$. Since the (translated) $t$-coordinate $r(t,\xi) = t -a $ is precisely the (absolute) distance function from the hypersurface $\Sigma_{a} = \{ a \} \times N \hookrightarrow M$ we have that
\[
| \nabla r | =1
\]
and the level sets
\[
\Sigma_{t + a} = r^{-1}(t) = \{t + a \} \times N,
\]
with $0 \leq t \leq b-a$, are parallel hypersurfaces. Moreover, the second fundamental form and the mean curvature of $\Sigma_{t}$ with respect to Gauss map $\vec \nu = \nabla r$ are given, respectively, by
\[
\mathrm{II}_{\Sigma_{t}} = \Hess(r)|_{\Sigma_{t}} =  \s'(t+a)\s(t+a) g^{N}
\]
and
\[
H_{\Sigma_{t}} = \Delta r = (m-1) \frac{\s'}{\s}(t+a).
\]
It follows that $r$ is an isoparametric function turning $\bar \O$ into an isoparametric domain. We note explicitly that each leaf $\Sigma_{t}$ is totally umbilical (namely, the traceless second fundamental form vanishes identically).

In case $(N,g^{N})$ is a compact Lie group endowed with a bi-invariant Riemannian metric, then the domain $\bar \Omega = [a,b] \times N$ inside $I \times_{\s} N$ is homogeneous with group $N$. Actually the same holds if $N=G/H$ is a homogeneous manifold.
\end{example}

\begin{example}[Euclidean homogenenous domains with non-compact leaves]\label{Ex:IsoparametricEuclidean}
Taking the Euclidean space $\mathbb{R}^n$ we easily obtain two different types of isoparametric domains with non-compact leaves:
\begin{itemize}
\item \textit{Cylindrical annuli:} consider the tube whose equidistants are the right cylinders $\{\Sigma_t\}_{t\in (a,b)}$ with axis given by a straight line $a$ through the origin $o\in \mathbb{R}^n$. Thanks to the isotropy of the Euclidean space, we can suppose that $a=\mathbb{R} \vec{e}_n=\mathbb{R}(0,...,0,1)$. Then, each leaf takes the form
\begin{displaymath}
\Sigma_t=\{(x',x_n)\in \mathbb{R}^n\ |\ x'\in \mathbb{S}^{n-1}_t,\ x_n\in \mathbb{R}\}
\end{displaymath}
for $\mathbb{S}^{n-1}_t$ the $(n-1)$-sphere of radius $t$, centred at the origin.

In this way we obtain an isoparametric foliation of the domain $\bar \Omega=\cup_{t\in [a,b]}\Sigma_t$ with leaves that have constant mean curvature equal to $H(\Sigma_t)=\frac{1}{t}$. A possible isoparametric function is
\begin{align*}
f(x_1,...,x_n)=\sqrt{x_1^2+...+x_{n-1}^2}=|x'|
\end{align*}

\item \textit{Slabs:} consider the tube whose equidistants are the hyperplanes $\{\Sigma_t\}_{t\in (a,b)}$ parallel to
\begin{displaymath}
\Sigma_0=\{x\in \mathbb{R}^n\ |\ x\cdot \vec{\nu}_0=0\}
\end{displaymath}
for a fixed vector $\vec{\nu}_0\in \mathbb{S}^{n-1}$.

As before, we can suppose $\vec{\nu}_0=\vec{e}_n$. Then, the leaves are 
\begin{displaymath}
\Sigma_t=\Sigma_0+t\vec{\nu}_0=\{(x',t)\ |\ x'\in \mathbb{R}^{n-1}\equiv \Sigma_0\}
\end{displaymath}
These hyperplanes give the domain $\bar \Omega=\cup_{t\in [a,b]}\Sigma_t$ an isoparametric structure, whose leaves have vanishing mean curvature. A possible isoparametric function is
\begin{align*}
f(x_1,...,x_n)=x_n
\end{align*}
\end{itemize}
In both cases, the domain $\bar \O$ is homogeneous with groups, respectively, $G=\mathbf{SO}(n)$ and $G=\mathbb{R}^{n-1}$.
\end{example}

\begin{example}[Generalized Hopf-Fibration]
Let $M = \mathbb{S}^3$ and $F(x)=x_1^2+x_2^2-x_3^2-x_4^2$ be the Cartan-Munzner polynomial that gives rise to Clifford tori $T(r) = \mathbb{S}^1(r) \times \mathbb{S}^1(\sqrt{1-r^2})$ with $0<r<1$. Then $F^{-1}([t_1,t_2])$ is a homogeneous domain by the action of $G= \mathbf{SO}(2) \times \mathbf{SO}(2)$. Similar examples can be constructed in the higher dimensional spheres $\mathbb{S}^n$, using the isoparametric functions $F(x)=l (x_1^2+...+x_k^2)-k(x_{k+1}^2+...+x_{n}^2)$ for $k+l=n+1$. Note that the leaves of these isoparametric domains are not totally umbilical (and, in particular, they have not a warped product structure of the form $I \times_\sigma N$).
\end{example}

\begin{example}[Cartan homogenenous domains]
 Tubes around tori are just one of the possible families of examples of homogenenous domains in the sphere $\ss^{m}$. For different choices of the Cartan-Munzner polynomial, corresponding to different choices of the Lie subgroup $G \subset \mathbf{SO}(m+1)$, we refer to \cite{Sh}. An account of more examples, in different ambient spaces, can be found in \cite{Va}.
\end{example}

\subsection{Weighted symmetric domains}
When formulated in the context of a weighted Riemannian manifold $M_{\Psi}$, the notion of isoparametric domain can be naturally generalized as follows.

Recall that, given a smooth hypersurface $\Sigma$ oriented by $\vec \nu$ inside the weighted manifold $M_{\Psi}$, its {\it weighted mean curvature} (in the sense of Gromov) $\vec H_{\Psi} = H_{\Psi} \vec \nu$ is given by
\[
H_{\Psi} = H - g(\nabla \Psi, \vec\nu)
\]
where $\vec H = H \vec{\nu}$ is the usual mean curvature vector field, i.e., the (unnormalized) trace of the second fundamental form.
\begin{definition}[$\Psi$-isoparametric domain]
 Let $M_{\Psi}$ be a weighted Riemannian manifold. We say that $\bar \O$ is a $\Psi$-isoparametric domain if $\bar \O$ is foliated by parallel hypersurfaces $\Sigma_{t}$ of constant weighted mean curvature. Equivalently, each leaf $\Sigma_{t}$ is the level set of a $\Psi$-isoparametric function $f$:
 \[
 |\nabla f | = \alpha(f) \quad \text{and} \quad \Delta_{\Psi} f = \beta (f).
 \]
\end{definition}
\smallskip

The notion of a homogeneous domain can be extended to the weighted setting using a similar spirit. In this case, however, it is not a-priori clear how to incorporate the weighted structure into the homogeneity condition. We choose to adopt the following
\begin{definition}[$\Psi$-homogenenous domain]
  Let $M_{\Psi}$ be a weighted Riemannian manifold. Say that $\bar \O$ is a $\Psi$-homogeneous domain if it is a $\Psi$-isoparametric domain and a homogeneous domain simultaneously. \\
  Equivalently, $\bar \O$ is $\Psi$-homogeneous if it is a homogeneous domain satisfying the ``weight compatibility condition''
\begin{align}\label{WeightCompatibilityCondition}
g(\nabla \Psi, \vec{\nu}) = const \ \ \textnormal{on each leaf } \Sigma_{t}
\end{align} 
\end{definition}
The equivalence of these two conditions come from the very definition of weighted mean curvature and the fact that a homogenenous domain has constant (ordinary) mean curvature.\smallskip

\begin{remark}[From homogenenous to $\Psi$-homogenenous]\label{rem: psi-homog}
It is worth noting that, if $P$ is the soul of $\bar \O$ and $d(x) = \dist(x,P)$, the natural choice $\Psi(x)=\hat{\Psi}(d(x))$ turns any(!) homogeneous domain into a $\Psi$-homogeneous domain. However, as we shall see, there are interesting $\Psi$-homogeneous domains that do not fall in this category. See Example \ref{Ex:IsoparametricGaussian}.
\end{remark}

\begin{example}
By definition of $\Psi$-symmetry and according to Remark \ref{rem: psi-homog}, Examples   \ref{ex: balls-homog} and \ref{Ex:WarpedProduct}  trivially generalize, respectively, to the case of weighted model manifolds and annuli in weighted warped product manifolds, up to assuming taht the weight has the form $\Psi(x) = \hat \Psi (d(x,o))$ and $\Psi(x) = \hat \Psi (\dist(x,\Sigma_{a}))$.
\end{example}

\begin{example}[Gaussian isoparametric domains with non-compact leaves]\label{Ex:IsoparametricGaussian}
Take the Gaussian space $\mathbb{G}^n$. The weighted mean curvature of a $\vec{\nu}$-oriented smooth hypersurface $\Sigma \subset \mathbb{G}^n$ is
\begin{displaymath}
H_\Psi = H-g(-x,\vec{\nu})=H+g(x,\vec{\nu})
\end{displaymath}
Using this fact, we can easily generalize the two examples obtained in \eqref{Ex:IsoparametricEuclidean}:
\begin{itemize}
\item \textit{Weighted cylindrical annuli:} As done in the non-weighted case, we consider
\begin{align*}
\Sigma_t=\{(x',x_n)\in \mathbb{R}^n\ |\ x'\in \mathbb{S}_t^{n-1},\ x_n\in \mathbb{R}\}
\end{align*}
for $\mathbb{S}^{n-1}_t$ the $(n-1)$-sphere of radius $t$, centred at the origin.

It follows that the normal vector field to the leaf $\Sigma_t$ is
\begin{displaymath}
\vec{\nu}_t(x)=\vec{\nu}_t\Big((x',x_n)\Big)=\frac{x'}{|x'|}\ \ \ \ \forall x\in \Sigma_t
\end{displaymath}
where we are identifying $x'$ with $(x',0)$. So
\begin{displaymath}
g(x,\vec{\nu}_t(x))=\frac{|x'|^2}{|x'|}=|x'|=t
\end{displaymath}
is constant on each $\Sigma_t$. Using this equality and the fact that the mean curvature of $\Sigma_t$ is $H(\Sigma_t)=\frac{1}{t}$, we obtain that
\begin{displaymath}
H_\Psi(\Sigma_t)=\frac{1}{t}+t
\end{displaymath}
constant on each $\Sigma_t$.

\item \textit{Weighted slabs:} As before, let $\vec{\nu}_0=\vec{e}_n$ and consider
\begin{displaymath}
\Sigma_t=\Sigma_0+t\vec{\nu}_0=\{(x',t)\ |\ x'\in \mathbb{R}^{n-1}\equiv \Sigma_0\}
\end{displaymath}
with normal vector field to $\Sigma_t$ given by
\begin{displaymath}
\vec{\nu}_t(x)=\vec{\nu}_t \Big((x',x_n)\Big)=\frac{(0,x_n)}{|x_n|}=\frac{t}{|t|} \vec{e}_n
\end{displaymath}
So
\begin{displaymath}
g(x,\vec{\nu}_t(x))=\frac{|x_n|^2}{|x_n|}=|x_n|=t
\end{displaymath}
and thus
\begin{displaymath}
H_\Psi(\Sigma_t)=H(\Sigma_t)+t=t
\end{displaymath}
constant on each $\Sigma_t$.
\end{itemize}
In particular, both weighted cylindrical annuli and weighted slabs are $\Psi$-homogeneous domains whose weight $\Psi$ is not symmetric.
\end{example}

\begin{example}[Gaussian-like weighted spaces]
Consider the weighted space $\mathbb{R}^n_\Psi=\Big(\mathbb{R}^n, g^{\mathbb{R}^n}, e^{-\Psi}dx\Big)$ for a symmetric weight $\Psi(x)=A|x|^2+B$ and $A,B\in \mathbb{R}$, $A\neq 0$. Then, the previous examples with non-compact leaves (parallel hyperplanes and coaxial cylinders) and the spherical tube shall continue to be $\Psi$-homogeneous domains.

Indeed, the gradient of the weight is
\begin{displaymath}
\nabla \Psi(x)=2Ax
\end{displaymath}
and following the previous calculations, we obtain that the weighted mean curvature of each equidistant of the above mentioned domains is constant.
\end{example}


\section{Symmetric functions}\label{section-symmetricfunctions}
Laid the foundations of the theory of isoparametric domains, we must specify what we mean by \textit{symmetry} when we talk about functions defined on them. Accordingly, one introduces the {\it average operator}
\begin{equation}\label{averageoper}
\CA_{\Psi} (u)(x) = \frac{1}{\area_\Psi \Sigma_{t(x)}} \int_{\Sigma_{t(x)}} u(y) \da_{\Psi}
\end{equation}
and put the following
\begin{definition}
Let $\bar \O$ be a compact weighted isoparametric domain inside the weighted manifold $M_{\Psi}$. Say that the function $u$ on $\bar \O$ is symmetric if
\[
u(x) = \CA_{\Psi} (u)(x).
\]
\end{definition}

\begin{remark}[Symmetry condition using distance function]
 If $\bar \O$ is a compact $\Psi$-isoparametric domain with soul $P$ and $d(x) = \dist(x,P)$, then the following are equivalent:
\begin{itemize}
\item [(a)]  $u=\CA_{\Psi} (u)$.
\item [(b)] $u(x)  = \hat{u}(d(x))$.
\end{itemize}
 The advantage of  characterization (b) over (a) is that it makes sense even if $P$ is non-compact and $u$ is not necessarily integrable on the leaves of the foliation.
\end{remark}

One of the main features of weighted isoparametric domains is that the corresponding average operator, that preserves the smoothness of functions, commutes with the weighted Laplacian. This property is formalized in the following Lemma that extends \cite[Proposition 13]{savo2018geometric} to the weighted setting.

\begin{lemma}[Savo]\label{lem: Commutation law}
Let $\O$ be a smooth, compact, weighted isoparametric domain with soul $P$ inside the weighted manifold $M_{\Psi}$. Let $\CA_{\Psi} $ be the average operator defined on $L^{1}(\O,\dv_{\Psi})$ by \eqref{averageoper}. Then the following hold:
 \begin{enumerate}
 \item [(a)] If $u  \in C^{k+2}(\O )$, then $\CA_{\Psi} (u) \in C^{k}(\O )$.
 \item [(b)] Given $u \in C^{4}(\O )$, $\CA_{\Psi}  ( \Delta_{\Psi} u ) = \Delta_{\Psi} \CA_{\Psi} (u)$.
\end{enumerate}
\end{lemma}

\begin{notation}
 For the sake of brevity, we  shall write condition (b) as the commutation rule
 \[
 [\CA_{\Psi} ,\Delta_{\Psi}]  = 0.
 \]
 A similar convention will be adopted during the paper for other operators.
\end{notation}
The proof is a minor variation of the original one in the Riemannian setting. For the sake of completeness, the details are supplied in the Appendix.

\subsection{Local vs global symmetry}
The notion of symmetry defined in the previous subsection can be formulated equivalently in terms of a first order condition.\smallskip

Let $\bar \O$ be an isoparametric domain with compact soul $P$ inside the weighted Riemannian manifold $M_{\Psi}$. We set, as usual, $d(x) = \dist(x,P)$ so that $\bar \O= \cup_{r\in [r_1,r_2]} \Sigma_{r}$  is foliated by the smooth, embedded, parallel hypersurface $\Sigma_{r} = \{ x \in M : d(x) = r\}$ in the same isotopy class.

\begin{definition}[Local symmetry]
Say that $u \in C^{1}(\bar \O)$ is symmetric at $x_{0} \in \bar \O$ if, for any smooth vector field $X$ on $\bar \O$ satisfying
\[
i)\ X|_{x_{0}} \not=0,\quad ii)\ g(X|_{x_{0}},\nabla d(x_{0}))= 0,
\]
it holds
\[
X(u)(x_{0}) = g(X|_{x_{0}},\nabla u(x_{0})) =0.
\]
In case $u$ is symmetric at every point $x \in \bar \O$ we say that $u$ is locally symmetric on $\bar \O$.
\end{definition}

\begin{remark}
 Clearly, the local symmetry at $x_{0}$ can be formulated in either of the following equivalent ways.
\begin{enumerate}
 \item [i)] Let $(\nabla u(x_{0}))^{\top}$ denote the orthogonal projection of $\nabla u(x_{0})$ on the tangent space $T_{x_{0}}\Sigma_{d(x_{0})}$. Then
 \[
 (\nabla u(x_{0}))^{\top} = 0.
 \]
 \item [ii)] The gradient of $u$ at $x_{0}$ is parallel to $\nabla d(x_{0})$:
 \[
\nabla u (x_{0}) \in  \mathrm{span}{\nabla d (x_{0})}  = (T_{x}\Sigma_{d(x_{0})})^{\perp}.
\]
\end{enumerate} 

\end{remark}

\begin{lemma}
 Keeping the above notation, the function $u$ is locally symmetric on $\bar \O $ if and only if $u$ is symmetric in the global sense, i.e., $u(x) =\hat{u}(d(x))$.
\end{lemma}

\begin{proof}
Assume that $u$ is locally symmetric and suppose by contradiction that there exist $r\geq 0$ and $x,y\in \Sigma_r$ such that $u(x)>u(y)$. Each leaf $\Sigma_r$ is connected, therefore we can consider a smooth immersed\footnote{a connected smooth manifold $N$ can be always endow with a complete Riemannian metric $h$. Therefore, any two given points $x,y \in N$ are connected by a minimizing $h$-geodesic, which is a smooth immersed curve of $N$.} curve $\gamma:[0,1]\to \Sigma_r$ joining $\g(0)= x$ to $\g(1)=y$. Since $u\circ \g$ is a $C^{1}$ function satisfying $u\circ \g(0) > u\circ \g(1)$, there exists $\bar t \in [0,1]$ such that
\[
g((\nabla u)(\g(\bar t)) , \dot \g(\bar t) ) =  \frac{d}{dt} (u\circ \gamma)(\bar t)<0.
\]
This contradicts the local symmetry because $0 \not= \dot \g(\bar t) \in T_{\g(\bar t)}\Sigma_{r}$.
\end{proof}


\section{Maximum principles, uniqueness and symmetry}\label{section-MaximumPrinciple-Uniqueness}\label{section-maximum-unique-symmetry}

Maximum principles for Schr\"odinger operators and uniqueness issues for solutions of semilinear PDEs permeate the whole theory of symmetry problems and the whole paper. Therefore, we devote this preliminary section to review briefly these topics both in the compact and in the non-compact settings.

\subsection{Compact maximum principle} In their book \cite[Section 5, Theorem 10]{PW},  Protter-Weinberger introduced a form of the Maximum Principle valid for elliptic operators in the presence of zeroth order terms. Their celebrated result states as follows.

\begin{proposition}[Compact Maximum Principle]\label{prop:CompactPWMaxPrinc}
Let $M_\Psi=(M,g, \dv_{\Psi})$ be a compact weighted Riemannian manifold with boundary $\partial M \not= \emptyset$ and suppose we are given on $M_\Psi$ the Schrödinger operator $\CL=\Delta_\Psi-q$, where $q\in C^0(M)$. Assume that there exists a function $\varphi \in C^0(M)\cap C^2(\inte(M))$ solution of the problem
\begin{align}\label{eq:CompactPhi}
\left\{\begin{array}{ll}
\CL \varphi \leq 0 & \inte M\\
\varphi >0 & M
\end{array}\right.
\end{align}
Then, any solution $u\in C^0(M)\cap W^{1,2}_\textnormal{loc}(\inte M )$ of
\begin{align*}
\left\{\begin{array}{ll}
\CL u \geq 0 & \inte M\\
u \leq 0 & \pM
\end{array}\right.
\end{align*}
satisfies $u \leq 0$ in $M$.
\end{proposition}
\begin{proof}
Consider the positive part of the function $u$
\begin{align*}
u_+=\max\{u,0\}
\end{align*}
Then $u_+$ satisfies
\begin{align*}
\left\{\begin{array}{ll}
\CL u_+ \geq 0 & \inte M\\
u_+ = 0 & \pM;
\end{array}
\right.
\end{align*}
see e.g.  \cite[Lemma 6.1]{PS} for a proof that works in the nonlinear setting. Defining the function $0\leq \omega=\frac{u_+}{\varphi}$ on the weighted manifold $M_{\Phi}$, where $\Phi = \log(\vp^{-2})+\Psi$, we get
\begin{align*}
\left\{
\begin{array}{ll}
\Delta_{\Phi}\omega \geq 0 & \inte M\\
\omega = 0 & \partial M,
\end{array}
\right.
\end{align*}
By the usual maximum principle we obtain $\omega\leq 0$ in $M$ that implies $\omega =0$ in $M$, i.e. $u_+ = 0$ in $M$, as claimed.
\end{proof}

Observe that for a compact Riemannian manifold with boundary $M$ there is no loss of generality in assuming that $M$ is a smooth bounded domain inside a closed Riemannian manifold $(N,g^{N})$; \cite[Theorem A]{PV}. Thus, the existence of a function $\varphi$ satisfying \eqref{eq:CompactPhi} is guaranteed under the assumption that $\l_{1}^{-\CL}(M) >0$. Indeed, in this case, once $q$ and $\Psi$ are extended with the same regularity to $N$, we can slightly enlarge $M$ to some smooth domain $\O \Subset N$ with $\l_1^{-\CL}(\O)> 0$ and take as $\varphi$ the restriction to $M$ of the first eigenfunction on $\O$. The existence of such a domain $\O$ could be seen as a trivial consequence of a deep continuity property of the Dirichlet eigenvalues with respect to the (Gromov-)Hausdorff convergence. See e.g. the paper \cite{Ch} by Chenais for the case of Hausdorff converging uniformly Lipschitz domains of the Euclidean space. However, one can obtain the existence of $\O$ using much more elementary considerations. We are going to provide the arguments for the sake of completeness.

\begin{lemma}\label{Lem:L1ContinuityEnlarging}
Let $N_\Psi=(N,g^{N},  \dv_\Psi)$ be a complete weighted Riemannian manifold (without boundary) and $\CL = \Delta_\Psi - q$ with $q \in C^{0}(N)$. Let $D \Subset N$ be a smooth  domain such that  $\l^{-\CL}_{1}(D)>0$. Then there exists a  smooth  domain $D \Subset \Omega\Subset N$ satisfying $\l_1^{-\CL}(\O)> 0$.
\end{lemma}
\begin{proof}
Consider a sequence of nested smooth domains $N\Supset \Omega_1 \Supset \Omega_2 \Supset ...\ \Omega_n \Supset\Omega_{n+1}\ ... \Supset  D $ satisfying $\bigcap_n \Omega_n = \bar D$ and let $Q_n$ and $Q$ be the quadratic forms associated to the Rayleigh quotient on $\O_n$ and on $D$ respectively
\begin{align*}
Q_n(u):=\int_{\O_n} \left(|\nabla u|^2 + q u^2 \right) \dv_\Psi,\ \ \ \ \ u \in W^{1,2}_0(\O_n, \dv_\Psi)\\
Q(u):= \int_{D} \left(|\nabla u|^2 + q u^2 \right) \dv_\Psi,\ \ \ \ \ u \in W^{1,2}_0(D, \dv_\Psi).
\end{align*}
By the domain monotonicity of the first Dirichlet eigenvalue we have
\[
\lambda^{-\CL}_1(D)\geq \lambda^{-\CL}_1(\Omega_n),\, \forall n\in \nn.
\]
Therefore, if $\{u_n\}_n \subset C^\infty(\bar\Omega_n)$ is the sequence of first Dirichlet eigenfunctions corresponding to $\l_{1}^{-\CL}(\Omega_{n})$, normalized so to have
\begin{align*}
\left\{
\begin{array}{l}
u_n \geq 0\ \ \ \ \ \textnormal{in}\ \O_n\\
\| u_n\|_{L^2(\Omega_n, \dv_\Psi)}=1,
\end{array}
\right.
\end{align*}
then, by extending each $u_n$ to 0 in $\Omega_1\setminus \Omega_n$ so that $u_{n} \in W^{1,2}_{0}(\Omega_{1})$, we get
\begin{displaymath}
\left\{ \begin{array}{l}
\| \nabla u_n \|^2_{L^2(\Omega_1, \dv_\Psi)}=\|\nabla u_n\|^2_{L^2(\Omega_n, \dv_\Psi)} \medskip \\
\| u_n\|_{L^2(\Omega_1, \dv_\Psi)}= \| u_n\|_{L^2(\Omega_n, \dv_\Psi)}=1 \medskip \\
Q_1(u_n)=Q_n(u_n)=\lambda^{-\mathcal{L}}(\O_n)\leq \lambda^{-\mathcal{L}}(D) .
\end{array}\right.
\end{displaymath}
In particular
\begin{align*}
\| \nabla u_n\|^2_{L^2(\O_1,\dv_\Psi)}&=\lambda_1^{-\mathcal{L}}(\O_n)- \int_{\O_n} q u_{n}^2\ \dv_\Psi\\
&\leq \lambda_1^{-\mathcal{L}}(D) + \| q \|_{L^\infty(\Omega_{1},\dv_\Psi)}.
\end{align*}
We have deduced that $\{u_n\}_n$ is a bounded sequence in $W^{1,2}_0(\Omega_1, \dv_\Psi)$. Then there exists a subsequence $\{u_{n_k}\}_{k}$ converging weakly in $W^{1,2}_0(\Omega_1, \dv_\Psi)$ and strongly in $L^{2}(\Omega_{1}, \dv_\Psi)$ to some function $v \in W^{1,2}_0(\Omega_1, \dv_\Psi)$. Clearly,
\[
\| v \|_{L^2(\Omega_1, \dv_\Psi)}=1.
\]
Moreover, since we can always assume that $u_{n_k}\xrightarrow[]{a.e.} v$ and, by assumption, $\bigcap_n \Omega_n = \bar D$, we have $ v = 0 \text{ a.e. on }\Omega_{1} \setminus \bar D$.
But, in fact, 
\[
v = 0 \text{ a.e. on }\Omega_{1} \setminus D
\]
because the smooth boundary $\partial D$ of $D$ has measure zero.
It follows from \cite[Proposition 2.11]{BG} that
\[
v\in W_0^{1,2}(D)
\]
and thus
\[
\lambda_1^{-\mathcal{L}}(D)\leq Q(v)=Q_{1}(v).
\]

Now, using the lower semicontinuity of the quadratic form $Q_{1}$ with respect to the weak $W^{1,2}$-topology,
we obtain
\begin{align*}
Q_{1}(v) & \geq \lambda^{-\CL}_1(D)\\
&\geq \limsup_k \lambda^{-\CL}_1(\Omega_{n_k})\\
&\geq \liminf_k \lambda^{-\CL}_1(\Omega_{n_k})\\
&=\liminf_k Q_1(u_{n_k})\\
&\geq Q_1(v),
\end{align*}
showing that
\[
\lim_k \lambda^{-\CL}_1(\Omega_{n_k}) = \lambda^{-\CL}_1(D) >0.
\]
The desired conclusion now follows  by choosing $\Omega = \Omega_{k_{0}}$ with $k_{0}$ large enough.
\end{proof}

As a consequence of Proposition \ref{prop:CompactPWMaxPrinc} and Lemma \ref{Lem:L1ContinuityEnlarging}, on noting also that if $\lambda_{1}^{-\CL}(\inte M)=0$ then the corresponding first Dirichlet eigenfunction $u \geq 0$ violates the maximum principle, we have the validity of the following well known characterization.

\begin{corollary}
 Let $M_{\Psi} = (M,g,\dv_{\Psi})$ be a compact weighted Riemannian manifold with smooth boundary. Then, the compact maximum principle of Proposition \ref{prop:CompactPWMaxPrinc} for the Schr\"odinger operator $\CL$ holds if and only if $\lambda_{1}^{-\CL}(\inte M)>0$.
\end{corollary}

When specified to the stability operator, the previous result takes the following form.

\begin{corollary}\label{Cor:CompactCase}
Let $M_\Psi=(M,g,\dv_{\Psi})$ be a compact weighted Riemannian manifold with smooth boundary $\partial M \not=\emptyset$. Assume that $u\in C^0(M)\cap C^2(\textnormal{int}(M))$ is a strongly stable solution of $\Delta_\Psi u=f(u)$ on $M$. If  $v\in C^0(M)\cap W^{1,2}_\textnormal{loc}(\textnormal{int}(M))$ satisfies
\begin{align*}
\left\{
\begin{array}{ll}
\Delta_\Psi v \geq f'(u) v & \inte M\\
v\leq 0 & \partial M
\end{array}
\right.
\end{align*}
then $v\leq 0$ on $M$.
\end{corollary}

\subsection{Non-compact maximum principle: parabolicity}\label{Section:Parabolicity-NonCaompactMP}

Let $M_{\Psi}$ be a (connected) weighted manifold with (possibly empty) boundary $\partial M$ and outward pointing unit normal $\vec \nu$. Say that $M_{\Psi}$ is \textit{Neumann-parabolic} ($\CN$-parabolic for short) if, for any given $v \in C^{0}(M) \cap W^{1,2}_{loc}(\inte M,  \dv_\Psi)$ satisfying
\[
\begin{cases}
 \Delta_{\Psi} v \geq 0 & \inte M\\
 \partial_{\vec \nu} v \leq 0 & \partial M\\
 \sup_{M} v < +\infty
\end{cases}
\]
it holds
\[
v \equiv const.
\]
Obviously, in case $\partial M = \emptyset$, the normal derivative condition is void.

As the definition shows, parabolicity is a kind of compactness from the viewpoint of the (weighted) Laplacian. This is also visible in the next theorem. Further instances will be presented in Section \ref{section-tools-potential}.

\begin{theorem}[Ahlfors maximum principle, \cite{IPS, ILPS}]\label{Th:AhlforsMP}
 If $M_{\Psi}$ is a $\CN$-parabolic weighted manifold with $\partial M \not = \emptyset$, then for any $v \in C^{0}(M) \cap W^{1,2}_{loc}(\inte M)$ satisfying
 \[
\begin{cases}
 \Delta_{\Psi} v \geq 0 & \inte M\\
 \sup_{M} v < +\infty
\end{cases}
\]
it holds
\[
\sup_{M} v = \sup_{\partial M} v.
\]
\end{theorem}

Using Theorem \ref{Th:AhlforsMP}, the proof of Proposition \ref{prop:CompactPWMaxPrinc} extends to the context of non-compact parabolic Riemannian manifolds: in addition, we only have to require suitable bounds on the functions $u$ and $\varphi$:

\begin{proposition}(Non-Compact Maximum Principle)\label{prop:NonCompactPWMaxPrinc}
Let $M_\Psi=(M,g, \dv_\Psi)$ be a $\CN$-parabolic weighted Riemannian manifold with boundary $\partial M \not= \emptyset  $ and set $\CL=\Delta_\Psi-q$ with $q\in C^0(M)$. Assume that there exists $\varphi \in C^2(M)$ satisfying
\begin{align}\label{eq:NonCompactPhi}
\left\{
\begin{array}{ll}
\mathcal{L}\varphi \leq 0 & \inte M\\
\frac{1}{C}\leq \varphi \leq C & M
\end{array}
\right.
\end{align}
for some constant $C\geq 1$. Then, any solution $u\in C^0(M)\cap W^{1,2}_\textnormal{loc}(\inte M)$ of
\begin{align*}
\left\{
\begin{array}{ll}
\mathcal{L}u\geq 0 & \inte M\\
u\leq 0 & \partial M\\
\sup_M u<+\infty
\end{array}
\right.
\end{align*}
satisfies $u\leq 0$ in $M$.
\end{proposition}
\begin{proof} Note that, thanks to the bounds on $\vp$, defining $\Phi=\log(\varphi^{-2})+\Psi$ as in the compact case, the weighted manifold $M_\Phi$ inherits the  $\CN$-parabolicity of $M_{\Psi}$. For instance, this can be seen by using the capacitary characterization of parabolicity as explained in \cite{IPS, ILPS}. Therefore, the proof of Proposition \ref{prop:CompactPWMaxPrinc} can  be carried out verbatim up to replacing the  classical maximum principle for the operator $\Delta_{\Phi}$ with the corresponding Ahlfors Maximum Principle of Theorem \ref{Th:AhlforsMP}.  
\end{proof}

\subsection{Uniqueness} 

It is well known that, for convex or concave nonlinearities, stable solutions of the corresponding semilinear equations on compact domains are (essentially) unique. More precisely, we recall the following result from \cite[Proposition 1.3.1]{Du}.

\begin{theorem}\label{th-uniqueness-Du}
 Let $M_\Psi=(M,g, \dv_\Psi)$ be a compact weighted Riemannian manifold with boundary components $(\partial M)_{j} \not=\emptyset$, $j=1,2$. Let $f:\rr \to \rr$ be a $C^{2}$ function satisfying either $f''(t) \leq 0$ or $f''(t) \geq 0$. Then, the boundary value problem
 \begin{align}\label{Eq:u}
\left\{ \begin{array}{ll}
\Delta_\Psi u=f(u)  & \inte M \\
u = c_{j}  \in \rr& (\partial M)_{j}
\end{array}
\right.
\end{align}
has at most one $C^2(M)$-stable solution unless $f(t)=-\lambda_1 t+c$, with $\lambda_1=\lambda_1^{-\Delta_\Psi}(M)>0$ the first Dirichlet eigenvalue. In this case, if $u_1$ and $u_2$ are two solutions, then $u_1-u_2=\alpha \varphi_1$, where $\alpha\in \mathbb{R}$ and $\varphi_1$ is a first Dirichlet eigenfunction of $-\Delta_\Psi$ on $M$. 
\end{theorem}

We are going to show how the proof of this uniqueness property extends to complete manifolds under a global Sobolev regularity condition. To this end, we first adapt  to complete manifolds with boundary the classical global Stokes theorem by Gaffney,  \cite{Ga}.

\begin{theorem}[Gaffney with boundary]\label{Th:Gaffney}
Let $M_\Psi=(M,g, \dv_\Psi)$ be a complete weighted Riemannian manifold with (possibly empty) boundary $\partial M$. Let $X$ be a vector field on $M$ such that:
\[
i)\, |X| \in L^{1}(M, \dv_\Psi),\,\, ii)\, \div_{\Psi} (X)\in L^1(M, \dv_\Psi),\,\, iii)\, g(X,\vec \nu)\in L^1(\partial M, \dv_\Psi),
\]
where $\vec \nu$ is the outward-pointing unit normal to $\partial M$. Then
\begin{align*}
\int_M \textnormal{div}_\Psi(X)\  \dv_\Psi=\int_{\pM} g(X,\vec \nu)\ \textnormal{da}_\Psi.
\end{align*} 
\end{theorem}
\begin{proof}
It is a consequence of the Riemannian extension property of complete manifolds that, even for manifolds with boundary, the completeness of $M$ implies the existence of a sequence of cutoff functions $\{\rho_k\}_k\subset C^\infty_c(M)$ satisfying
\begin{align}\label{cutoffs}
\left\{
\begin{array}{l}
0\leq \rho_k \leq 1\smallskip \\
||\nabla \rho_k||_{L^\infty(M, \dv)}\to 0 \smallskip\\
\rho_k \nearrow 1.
\end{array}\right.
\end{align}
See \cite[Page 16]{PV}. Since the vector field $\rho_k X$ is compactly supported, by the classical (weak) divergence theorem we have
\begin{align*}
\int_M \div_\Psi(\rho_k X)\  \dv_\Psi&=\int_{\pM} g(\rho_k X, \vec{\nu})\ \textnormal{da}_\Psi.
\end{align*}
On the other hand,
\begin{align*}
\int_M \div_{\Psi}(\rho_k X)\  \dv_\Psi=\int_M g(\nabla \rho_k, X)\  \dv_\Psi+\int_M \rho_k \div_{\Psi}(X)\  \dv_\Psi.
\end{align*}
Whence, we obtain
\begin{align}\label{Eq:Gaffney1}
\int_{\pM} g(\rho_k X, \vec \nu)\ \textnormal{da}_\Psi=\int_M g(\nabla \rho_k, X)\  \dv_\Psi+\int_M \rho_k \div_{\Psi}(X)\  \dv_\Psi.
\end{align}
To conclude the validity of \eqref{Eq:Gaffney1} we take the limit as $k\to+\infty$ once we have noted that, by dominated convergence,
\begin{align*}
\int_M \rho_k \div_{\Psi} (X)\  \dv_\Psi\to \int_M \div_\Psi (X)\  \dv_\Psi
\end{align*}
and
\begin{align*}
\int_{\pM} g(\rho_k X, \vec \nu)\ \textnormal{da}_\Psi\to \int_{\pM} g(X, \vec\nu)\ \textnormal{da}_\Psi
\end{align*}
while
\begin{align*}
\Bigg|\int_M g(\nabla \rho_k,X)\  \dv_\Psi \Bigg|\leq ||\nabla \rho_k||_{L^\infty(M, \dv)}||X||_{L^1(M, \dv_\Psi)}\to 0.
\end{align*}
\end{proof}

Using this global divergence theorem, we can now extend to complete manifolds the uniqueness result of Theorem \ref{th-uniqueness-Du}.

\begin{theorem}\label{Th:NoneEuclideanDupaigne}
Let $M_\Psi=(M,g, \dv_\Psi)$ be a complete weighted Riemannian manifold with boundary $\partial M \not= \emptyset$, and $u_1,u_2\in C^0(M)\cap W^{1,2}(\inte M, \dv_\Psi)\cap L^\infty (M)$ be stable solutions of \eqref{Eq:u} with $f\in C^1$ concave (or convex). Then $u_1=u_2$ unless $f(t)=At+B$ for some $A,B\in \mathbb{R}$.
\end{theorem}
\begin{proof}
Observe that $\omega=u_2-u_1$ solves
\begin{align}\label{Eq:w}
\left\{
\begin{array}{ll}
\Delta_\Psi \omega =f(u_2)-f(u_1) & \textnormal{in int}M\\
\omega=0 & \textnormal{on}\ \partial M,
\end{array}
\right.
\end{align}
Let $\omega_+ =\max(\omega,0)\in W^{1,2}(\inte M)\cap C^{0}(M)$. Using a standard approximation argument that relies on the completeness of $M$, we easily see that
\[
\omega_+ \in W^{1,2}_0(\inte M).
\]
Indeed, let $\{\rho_k\}_k\subset C^\infty_c(M)$ be the sequence of cutoff functions introduced in Theorem \ref{Th:Gaffney} and consider the corresponding sequence $\{\varphi_k=\rho_k \omega_+\}_k \subset W^{1,2}_0 (\inte M)$. Since, by dominated convergence, $\varphi_{k} \overset{L^{2}}{\longrightarrow} \omega_+$ and, moreover,
\begin{align*}
\int_M |\nabla &(\varphi_k - \omega_+)|^2\ \dv_\Psi \\
& \leq  \underbrace{2 \int_M |\omega_+|^2 |\nabla \rho_k|^2 \dv_\Psi}_{\xrightarrow{DCT}0} + \underbrace{ 2\int_M (1-\rho_k)^2 |\nabla \omega_+|^2\ \dv_\Psi}_{\xrightarrow{MCT}0} \longrightarrow 0
\end{align*}
we have $\varphi_k \xrightarrow{W^{1,2}} \omega_+ $. The claimed property thus follows form the fact that  $W^{1,2}_0(\inte M)$ is a closed subspace of $W^{1,2}(\inte M)$.

Now consider the vector field $X = \omega_+ \nabla \omega_+$. By the very definition, $X$ and $\div_\Psi (X)$ are $L^1$-functions and $X$ vanishes on the boundary $\partial M$. Thus, we can apply Theorem \ref{Th:Gaffney} obtaining
\begin{align}\label{Eq:w+}
\int_M{|\nabla \omega_+|^2\ \dv_\Psi}=-\int_M \Big( f(u_2)-f(u_1)\Big) \omega_+\  \dv_\Psi.
\end{align}
On the other hand, since $u_2$ is a stable solution, using $\vp_{k} = \rho_{k}\omega_{+}\in W_0^{1,2}(M,\dv_\Psi)$ as test functions in the stability condition, we obtain
\begin{align*}
\int_M|\nabla\vp_{k}|^2\  \dv_\Psi \geq -\int_M f'(u_2) \vp_k^2\  \dv_\Psi
\end{align*}
where
\begin{align*}
\int_M |\nabla \vp_k|^2\  \dv_\Psi =& \underbrace{\int_M \rho_k^2 |\nabla \omega_+|^2\  \dv_\Psi}_{\xrightarrow[]{MCT}\int_M |\nabla \omega_+|^2\  \dv_\Psi} +  \underbrace{\int_M \omega_+^2 |\nabla \rho_k|^2\  \dv_\Psi}_{\xrightarrow[]{DCT}0} \\
\ \\
&+  \underbrace{2 \int_M \rho_k \omega_+ <\nabla\rho_k, \nabla \omega_+>\  \dv_\Psi}_{=c_k}
\end{align*}
and
\begin{align*}
|c_k|\leq 2 \underbrace{\Bigg(\int_M \rho_k^2 |\nabla \omega_+|^2 |\nabla\rho_k|^2\  \dv_\Psi\Bigg)^\frac{1}{2}}_{\xrightarrow[]{DCT}0} \Bigg( \int_M \omega_+^2\  \dv_\Psi\Bigg)^\frac{1}{2}.
\end{align*}
Thus
\begin{align*}
\int_M |\nabla \vp_k|^2\  \dv_\Psi\to \int_M |\nabla\omega_+|^2\  \dv_\Psi.
\end{align*}
Moreover
\begin{align*}
-\int_M f'(u_2) \vp_k^2\  \dv_\Psi = -\int_M f'(u_2) \rho_k^2 \omega_+^2\  \dv_\Psi \xrightarrow[]{DCT}-\int_M f'(u_2) \omega_+^2\  \dv_\Psi.
\end{align*}
It follows that
\begin{align*}
\int_M |\nabla \omega_+|^2\  \dv_\Psi \geq -\int_M f'(u_2)\omega_+^2\  \dv_\Psi
\end{align*}
and this latter, together with \eqref{Eq:w+}, implies
\begin{align*}
-\int_M f'(u_2) \omega_+^2\  \dv_\Psi \leq -\int_M \Big(f(u_2)-f(u_1) \Big)\omega_+\  \dv_\Psi
\end{align*}
i.e.
\begin{align*}
\int_M \Big(f(u_2)-f(u_1)-f'(u_2)\omega_+\Big)\omega_+\  \dv_\Psi\leq 0.
\end{align*}
Since, by concavity, the above integrand is non-negative we deduce that
\[
\Big(f(u_2)-f(u_1)-f'(u_2)\omega_+\Big)\omega_+ = 0
\]
and two possibilities can occur: either $f(t)$ is strictly concave and, hence, $w_{+} \equiv 0$, or $f(t)$ is affine. Clearly, in the first case, $u_{2}\leq u_{1}$ and by reversing the role of $u_{1}$ and $u_{2}$ we conclude $u_{1} = u_{2}$ as desired.
\end{proof}

\subsection{Symmetry via average}\label{section-symmetrization-linear}
As a warm-up for the investigations of the paper we observe that, clearly, if the boundary value problem at hand
\begin{equation}\tag{\ref{DP}}
\begin{cases}
\Delta_{\Psi} u = f(u) & \text{in }\O  \\
u = c_{j} \in \rr& \text{on }(\partial \O)_{j}
\end{cases}
\end{equation}
has a unique solution, and we are able to construct at least one symmetric solution, then we are done. This happens e.g. in the affine setting $f(t) = At+B$. Indeed, the equation is clearly preserved by the average procedure, hence a symmetric solution exists. In order for the maximum principle to hold, we just need to assume that either $A\geq 0$ or, more generally, that $\Omega$ is small enough in the spectral sense, i.e.  $\l_{1}^{-\Delta_{\Psi} + A}(\O) > 0$. Thus, any solution of the corresponding Dirichlet problem \eqref{DP} is automatically strictly stable. This is the simplest situation that can occur.\medskip

\begin{proposition}\label{prop-symmetry-uniqueness}
 Let $M_{\Psi}$ be a weighted manifold  and let $\bar \O $ be a smooth, compact,  $\Psi$-isoparametric domain. The connected components of its boundary are denoted by $(\partial \O )_{j}$, $j=1,2$.

Let $u \in C^{\infty}(\O ) \cap C^{0}(\bar \O  )$ be a strictly stable solution of the problem
 \begin{equation}\label{Dirichlet-isoparametric}
\begin{cases}
 \Delta_{\Psi} u = Au+B & \text{in }\O  \\
 u = c_{j} & \text{on }(\partial \O )_{j}
\end{cases}
 \end{equation}
where $B,c_{j} \in \rr$. Then, $u$ is symmetric.
\end{proposition}

\begin{proof}
Using the commutation rule $[\CA_{\Psi},\Delta_{\Psi}]=0$ we see that the smooth function
\[
w = u - \CA_{\Psi}(u)
\]
solves the problem
\[
\begin{cases}
 \Delta_{\Psi} w = A w & \text{in }\O  \\
 w = 0 & \text{on }\partial \O.
\end{cases}
\]
The  maximum principle yields $w=0$ which means
\[
u = \CA_{\Psi}(u)\quad \text{on }\O
\]
as desired.
\end{proof}


\section{Symmetry of solutions on $\Psi$-homogeneous domains}\label{section-geometicDupaigne}

The main result of the section is a geometric interpretation of the arguments in \cite[Proposition 1.3.4]{Du}. The original symmetry result, for rotationally symmetric domains in the Euclidean spaces, is proved in \cite[Lemma 1.1]{AB}.

\begin{theorem}\label{th-weighed-Doupaigne}
Let $\bar \O$ be a compact $\Psi$-homogeneous domain with soul $P$ inside the weighted manifold $M_{\Psi}$. If $\mathcal{D}=\{X_1,...,X_k\}$ is an integrable distribution of Killing vector fields associated to the foliation of $\bar \O$, suppose that $\Psi$ satisfies the compatibility condition
 \begin{align}\label{Eq:KillingConstantAngle}
 g(X_i,\nabla \Psi) \equiv const \quad \text{on }\O,
 \end{align}
for every $i=1,...,k$.

Then, a stable solution $u \in C^{3}(\O ) \cap C^{1}(\bar \O )$ of
\begin{equation}\label{DP-tubes}
\begin{cases}
 \Delta_{\Psi} u = f(u) & \O \\
 u = c_{j} & (\partial \O )_{j}
\end{cases}
\end{equation}
is symmetric if and only if at least one of the following conditions hold:
\begin{enumerate}
\item [a)] \label{Cond:Killing1} $g(\nabla \Psi, X_i)\equiv 0$ for every $i=1,...,k$, i.e. $\Psi(x)=\widehat{\Psi}(\textnormal{dist}(x,P))$ is symmetric;
\item [b)]\label{Cond:Killing2} the mean value of $u$ over $\bar \O$ is zero.
\end{enumerate}
\end{theorem}

\begin{remark} For a Killing vector field $X$, condition \eqref{Eq:KillingConstantAngle} can be seen as a $\Psi$-compatibility property. Indeed, since $\div(X)=0$,
\begin{gather*}
g(X,\nabla \Psi) \equiv const\\
\Updownarrow\\
\div_\Psi(X)=\div(X)-g(X,\nabla \Psi)\equiv const.
\end{gather*}
Thus, in condition \eqref{Eq:KillingConstantAngle}, we are requiring that the divergence-free property of the Killing field $X$ is (in a certain sense) inherited by the weighted manifold.
\end{remark}

The proof of Theorem \ref{th-weighed-Doupaigne} relies on the fact that ($\Psi$-)Killing vector fields well behave with respect to the (weighted) Laplace-Beltrami operator. We first recall the following known characterization. 

\begin{lemma}\label{lemma-Killing}
 Let $(M,g)$ be a Riemannian manifold. Then, the vector field $X$ is Killing if and only if the commutation rule $[\Delta , X]=0$ holds. This means that, for any smooth function $u$, $\Delta X(u) = X(\Delta u)$.
\end{lemma}

\begin{proof}
 See \cite{FMV} for a computational proof that involves generic vector fields. On the other hand, following V. Matveev, the commutation rule can be also deduced directly from the fact that the flow of a Killing vector field is an infinitesimal isometry. Conversely, if the commutation rule holds then the flow of $X$ preserves the Laplacian and the Laplacian determines uniquely the Riemannian metric.
\end{proof}

In the special case of a Killing vector field tangential to the leaves of a weighted isoparametric domain, the commutation extends to the weighted Laplacian. This is a special case of the following

\begin{lemma}\label{lemma-weightedcomm}
 Let $M_{\Psi}$ be a  weighted manifold. If $X$ is a Killing vector field satisfying condition \eqref{Eq:KillingConstantAngle}, then
 \[
[ \Delta_{\Psi} , X ] = 0,\quad \text{on } \O 
 \]
 in the sense that, for any smooth function $u$ on $\O $,
 \[
\Delta_{\Psi}X(u) =  X(\Delta_{\Psi} u).
 \]
\end{lemma}

\begin{proof}
 Recall that
 \[
 \Delta_{\Psi} u = \Delta u - g(\nabla \Psi,\nabla u)
 \]
 and that, since $X$ is Killing,
 \[
 [\Delta , X ] = 0.
 \]
 Therefore, we are reduced to verify that
\begin{equation}\label{eq-weightedcomm}
 g(\nabla \Psi , \nabla X(u)) = D_{X}g(\nabla \Psi , \nabla u).
\end{equation}
To this end, let us start by computing
\begin{align*}
 g(\nabla \Psi , \nabla X(u) ) &= g(\nabla \Psi , \nabla g(X,\nabla u)) \\
 &= D_{\nabla \Psi} g(X,\nabla u) \\
 &= g(D_{\nabla \Psi} X, \nabla u) + g(X,D_{\nabla \Psi} \nabla u)\\
 &= -g(D_{\nabla u} X, \nabla \Psi) + \Hess(u)(X,\nabla \Psi),
\end{align*}
where in the last equality we have used that $X$ is Killing and the definition of the Hessian tensor. Now
\begin{align*}
 g(X,\nabla \Psi) = const &\Longrightarrow D_{\nabla u} \, g(X,\nabla \Psi) = 0 \\
 &\Longrightarrow g(D_{\nabla u} X , \nabla \Psi) + g(X, D_{\nabla u} \nabla \Psi) = 0\\
&\Longrightarrow -g(D_{\nabla u} X, \nabla \Psi) = \Hess(\Psi) (X,\nabla u).
\end{align*}
Inserting into the above gives
\begin{equation}\label{eq-weightedcomm1}
  g(\nabla \Psi , \nabla X(u) )  =  \Hess(u)(X,\nabla \Psi) + \Hess(\Psi) (X,\nabla u) .
\end{equation}
On the other hand,
\begin{align}\label{eq-weightedcomm2}
 D_{X} g(\nabla \Psi , \nabla u) &= g(D_{X}\nabla \Psi , \nabla u) + g(\nabla \Psi,D_{X}\nabla u) \\ \nonumber
 &= \Hess(\Psi)(X,\nabla u) + \Hess(u)(X,\nabla \Psi) .
\end{align}
Putting together \eqref{eq-weightedcomm1} and \eqref{eq-weightedcomm2} we conclude the validity of \eqref{eq-weightedcomm} as desired.
\end{proof}

We are now in the position to give the

\begin{proof}[Proof of Theorem \ref{th-weighed-Doupaigne}]
Consider a distribution $\mathcal{D}=\{X_{1},\cdots,X_{k} \}$ of Killing vector fields tangential to the leaves of the foliation and satisfying $g(\nabla \Psi, X_i)=const$ for every $i=1,...,k$. Let $X = X_{j}$ and define
 \[
 v = X(u) = g(\nabla u, X).
 \]
Since $u$ is locally constant on $\partial \O$ and $X|_{\partial \O}$ is tangential to $\partial \O$, we have
 \[
 v = 0\quad \text{on }\partial \O.
 \]
On the other hand, by Lemma \ref{lemma-weightedcomm} we deduce that
\[
\Delta_{\Psi} v = X(\Delta_{\Psi} u) = X(f(u)) = f'(u) X(u) = f'(u) v.
\]
It follows that $v \in C^{2}(\O)$ is a solution of the problem
\[
\begin{cases}
 \Delta_{\Psi} v = f'(u) \, v & \O \\
 v = 0 & \partial \O.
\end{cases}
\]
In particular, $\l_{1}^{-\Delta_{\Psi} + f'(u)}(\O) = 0$ and $v$ is a first eigenfunction corresponding to this Dirichlet eigenvalue. By the nodal domain theorem,
\[
v \geq 0.
\]
We are going to prove that the validity of at least one of the conditions a) or b) is equivalent to
\begin{equation}\label{zeromean}
\int_{\O} v \, \dv_{\Psi} =0
\end{equation}
and, hence to
\[
v\equiv 0.
\]
To this end, we use the $\Psi$-divergence theorem with  the vector field $Z = uX$. Since $\div X =0$ and $X_{x}$ is tangential to $\Sigma_{d(x)}$, on the one hand we have
\begin{align*}
 \int_{\O} \div_{\Psi} Z \, \dv_{\Psi} &= \int_{\O} g(\nabla u , X) \, \dv_{\Psi}+ \int_{\O} u \div_{\Psi} X\, \dv_{\psi} \\
 &=  \int_{\O} v \, \dv_{\Psi} + \int_{\O} u\div X \, \dv_{\Psi}- \int_{\O}  u \, g(\nabla \Psi , X) \dv_{\Psi}\\
 &= \int_{\O}v \, \dv_{\Psi}- g(\nabla \Psi, X) \int_{\O} u\ \dv_\Psi.
\end{align*}
On the other hand,
\begin{align*}
 \int_{\O} \div_{\Psi} Z \, \dv_{\Psi}= \int_{\partial \O} g(Z, \vec{\nu}) \da_{\Psi}= \int_{\partial \O} u \, g(X,\pm \nabla d) \da_{\Psi} = 0,
\end{align*}
where $d(x)=\textnormal{dist}(x,P)$. By putting together these two expressions we obtain
\begin{align*}
\int_{\O} v\ \dv_\Psi=g(\nabla \Psi, X) \int_{\O} u\ \dv_\Psi
\end{align*}
That is, \eqref{zeromean} holds if and only if either $g(\nabla \Psi, X)\equiv 0$ or $u$ has vanishing integral.

We have thus proved that if at least one of the conditions a) and b) is satisfied, then
\[
X_{j}(u)(x_{0}) = 0,\quad \forall j=1,\cdots,k,\ \ \ \forall x_0 \in \bar \O.
\]
Thanks to the fact that $\{X_{1}|_{x_{0}},\cdots,X_{k}|_{x_{0}}\}$ generates $T_{x_{0}}\Sigma_{d(x_{0})}$, this implies that $u$ is locally symmetric, and hence symmetric, on $\bar \O$. The proof of Theorem \ref{th-weighed-Doupaigne} is completed.
\end{proof}


\section{Symmetry of solutions in a non-homogeneous case}\label{SymmetryWithoutKillingFields}

In this section we discuss a case where we cannot apply Theorem \ref{th-weighed-Doupaigne} due to the absence of enough (if any) Killing vector fields tangential to the leaves of the tube. In fact, recall that, in nonpositive curvature, Killing fields tangential to the (concave) boundary of a domain are trivial as the following classical  theorem shows; see \cite {Ya}.

\begin{theorem}[Weighted Yano-Bochner]\label{th-YB}
 Let $M_\Psi=(M,g, \dv_\Psi)$ be a compact weighted Riemannian manifold with (possibly empty) concave boundary $\partial M$. This means that, if $\vec \nu$ denote the outer unit normal to $\partial M$, then $\mathrm{II}(Z,Z) = g(D_{Z}(-\vec \nu),Z) \geq 0$ for every $Z \in T\partial M$. Assume also that $\ric_\Psi=\ric+\Hess \Psi \leq 0$. 
 
Then, every Killing vector field $X$ on $M$ such that $X|_{\partial M} \in T\partial M$ and satisfying $\div_\Psi (X)\equiv const$ must be parallel. In particular, $|X| \equiv const$. Moreover, if $\ric_{\Psi}< 0$ at some point, then $X = 0$.
\end{theorem}

\begin{proof}
The weighted version of Bochner formula for Killing vector fields satisfying $\div_\Psi (X)\equiv const$ states that
 \[
 \frac{1}{2} \Delta_\Psi |X|^{2} = |D X|^{2} - \ric_\Psi(X,X).
 \]
Therefore, using the curvature assumption,
 \[
 \Delta_\Psi |X|^{2} \geq 0.
 \]
 By the Killing condition and the fact that $X|_{\partial \O}$ is tangential to $\partial \O$ we get
\begin{align*}
\partial_{\vec \nu} |X|^{2} &= -2\mathrm{II}(X,X),\quad \text{on }\partial \O.
\end{align*}
It follows that $v = |X|^{2}$ is a solution of the problem
\[
\begin{cases}
 \Delta_\Psi v \geq 0 & \O \\
 \partial_{\vec \nu} v = -2 \mathrm{II}(X,X) \leq 0 & \partial \O.
\end{cases}
\]
By the Hopf Lemma, $v \equiv const$. Using this information into the Bochner formula  gives that $|DX| =0$, i.e. $X$ is parallel, and $\ric_\Psi(X,X) = 0$.
\end{proof}

\begin{remark}\label{GeneralizedWeightedBochner}
For a general Killing vector field, without any request on the $\Psi$-divergence, the weighted Bochner formula states that
\begin{align*}
\frac{1}{2} \Delta_\Psi |X|^{2} = |D X|^{2} - \ric_\Psi(X,X)+Xg(X,\nabla \Psi)
\end{align*}
or, equivalently,
\begin{align*}
\frac{1}{2} \Delta_\Psi |X|^{2} = |D X|^{2} - \ric_\Psi(X,X)+g(X,\nabla \div_\Psi (X))
\end{align*}
Thus, the previous Theorem can be slightly generalised to Killing vector fields tangent to the boundary of the manifold and satisfying
\begin{align*}
g(X,\nabla \div_\Psi(X))\geq 0
\end{align*}
\end{remark}

\begin{remark}
Formally, the conclusion of Theorem \ref{th-YB} can be extended to Killing fields of bounded length on a complete Riemannian manifold with boundary and with quadratic volume growth. See Sections \ref{Section:Parabolicity-NonCaompactMP} and  \ref{section-tools-potential}.
\end{remark}

\begin{example}\label{ex-noKilling}
Take the annulus $A(-1,+1)=[-1,+1] \times N$ inside the Riemannian warped cylinder $M = \rr \times_\sigma N$ where:
\begin{enumerate}
 \item [i)] $(N,g^{N})$ is compact, $\partial N = \emptyset$,  and $\sect^{N}\equiv - k^{2}<0$;
 \item [ii)] $\s'(-1) \leq 0$, $\s'(+1)\geq 0$;
 \item [iii)] $\s''(r) \geq 0$ in $[-1,1]$.
\end{enumerate}
We have already observe in Example \ref{Ex:WarpedProduct} that $A(-1,1)$ is an isoparametric domain with totally umbilical leaves $\Sigma_{t} = \{ t \} \times N$, $-1 \leq t \leq 1$. In particular,
\[
\mathrm{II}_{\Sigma_{\pm}} = \pm  \s'(\pm 1)\s(\pm 1) g^{N}.
\]
It follows from ii) that
\begin{enumerate}
 \item [a)] $\partial A(-1,1) = \Sigma_{\pm 1}$ is concave.
\end{enumerate}
Morever, recalling that
\[
\sect_{M}(X\wedge Y) =
\begin{cases}
 0 & X,Y = \nabla r \medskip \\
 -\frac{\s''(r)}{\s(r)} & X = \nabla r,\, Y \in TN \medskip\\
\frac{ -k^{2} - \s'(r)^{2}}{\s(r)^{2}} & X,Y \in TN
\end{cases}
\]
by iii) we have
\begin{enumerate}
 \item [b)] $\sect_{M}<0$.
\end{enumerate}
An application of Theorem \ref{th-YB} gives that any Killing vector field $X$ of $\bar A(-1,1)$ tangential to $\partial  A(-1,1)$ must vanish identically.
\end{example}

As we are going to show, in the situation of Example \ref{ex-noKilling} we are still able to deduce a symmetry result. But there is a prize to pay: beside the assumption that the solution of the boundary value problem is (strictly) stable, the nonlinearity $f(t)$ has to be concave. In particular, when the fibre $N$ is compact, we are in the regime of uniqueness of the solution; see Theorem \ref{th-uniqueness-Du}. Despite of this drawback, on the one hand, it is not clear how to produce a-priori a symmetric solution (clearly, average does not work) and, on the other hand, the method we use works in a more general setting where, apparently, the non-compact uniqueness result of Theorem \ref{Th:NoneEuclideanDupaigne} is not applicable. See Remark \ref{rem-th-noKilling}.

\subsection{A non-compact symmetry result: statement and comments}

Let $M_\Psi=(M,g^{M}, \dv_\Psi)$ be the $m$-dimensional weighted Riemannian manifold given as the warped product
\[
M = I \times_{\s} N
\] where $(N,g^{N})$ is a  possibly non-compact $(m-1)$-dimensional Riemannian manifold with $\partial N = \emptyset$,  $I \subseteq \rr$ is an  interval, $\s:I \to \rr_{>0}$ is a smooth function and
\begin{align}\label{Eq:WeightForm}
\Psi(r,\xi)=\Phi(r)+\Gamma(\xi)
\end{align}
splits into the sum of two smooth functions depending respectively on the $I$-variable and on the $N$-variable. Consider the annulus $\bar A(r_1,r_2) = [r_1,r_2] \times N$. By the coarea formula, the volume of $\bar A(r_1,r_2)$ has the expression 
\[
\vol_\Psi(\bar A(r_1,r_2)) = \vol_\Gamma(N) \int_{r_1}^{r_2} e^{-\Phi(r)}\ \s^{m-1}(r)\ \dr.
\]
Moreover, we note explicitly that
\[
\DeltaM u = \partial^{2}_{r}u+(m-1)\frac{\s'}{\s} \partial_{r}u + \frac{1}{\s^{2}}\DeltaN u
\]
and thus
\begin{align*}
\Delta^M_\Psi u &= \partial^{2}_{r}u+(m-1)\frac{\s'}{\s} \partial_{r}u + \frac{1}{\s^{2}}\DeltaN u-g(\nabla^M u,\nabla^M \Psi)\\
&=\partial^{2}_{r}u+\Big((m-1)\frac{\s'}{\s}-\Phi'\Big) \partial_{r}u + \frac{1}{\s^{2}}\DeltaN u-\sigma^2 g^N \Bigg(\frac{\nabla^N u}{\s^2}, \frac{\nabla^N \Gamma}{\s^2}\Bigg)\\
&=\partial^{2}_{r}u+\Big((m-1)\frac{\s'}{\s}-\Phi'\Big) \partial_{r}u + \frac{1}{\s^{2}}\DeltaN u-\frac{1}{\sigma^2} g^N (\nabla^N u, \nabla^N \Gamma)\\
&=\partial^{2}_{r}u+\Big((m-1)\frac{\s'}{\s}-\Phi'\Big) \partial_{r}u + \frac{1}{\s^2} \Delta^{N}_\Gamma u
\end{align*}
In particular, $\bar A (r_1,r_2)$ is $\Psi$-isoparametric and we have the validity of the commutation rule
\begin{equation} \label{Eq:CommutatorLaplacian}
\begin{split}
[\Delta^M_\Psi , \Delta^N_\Gamma]=0.
\end{split}
\end{equation}

We are now ready to state our non-compact symmetry result. Since the underlying manifold is always $M_\Psi$ and there is no danger of confusion, from now on we shall omit the overscript $M$ in the corresponding quantities and operators.

\begin{theorem}\label{th-annuli}
Let $M_\Psi = (I \times_{\s} N)_\Psi$ where $(N,g^{N})$ is a complete (possibly non-compact), connected, $(m-1)$-dimensional Riemannian manifold with finite $\Gamma$-volume $\vol_{\Gamma}(N)<+\infty$.\smallskip
 
Let $u \in C^{4}(\bar {A}(r_{1},r_{2}))$ be a solution of  the Dirichlet problem
\begin{equation}\label{dir1}
\begin{cases}
 \Delta_\Psi u = f (u) & \text{in }A(r_{1},r_{2}) \\
 u \equiv c_{1} & \text{on }\{ r_{1} \}\times N \\
 u \equiv c_{2} &\text{on }\{r_{2}\} \times N.
\end{cases}
\end{equation}
where $c_{j} \in \rr$ are given constants and the function $f(t)$ is of class  $C^{2}$ and satisfies $f''(t) \leq 0$. If
\begin{equation}\label{radialnorm}
\| u \|_{C^{2}_{rad}} := \sup_{A(r_{1},r_{2})} |u | + \sup_{A(r_{1},r_{2})} |\partial_{r}u| + \sup_{A(r_{1},r_{2})} |\partial^{2}_{r}u| < +\infty,
\end{equation}
and $f'(u)\geq -B$, for some constant $B\geq 0$ satisfying
\begin{align}\label{Eq:BCondition}
0 \leq B < \left( \int_{r_1}^{r_2} \frac{\int_{r_1}^s e^{-\Phi(z)}\sigma^{m-1}(z)\ \textnormal{d}z}{e^{-\Phi(s)}\sigma^{m-1}(s)}\ \ds \right)^{-1}
\end{align}
then $u(r,\xi) = \hat u(r)$ is symmetric.
\end{theorem}

\begin{remark}
Under the additional assumption $[\Delta_\Psi , \Delta^N_\Gamma](u)\leq 0$, this symmetry result can be easily generalized to every smooth weight $\Psi(r,\xi)$ satisfying the condition $\partial_r \Psi \in L^\infty(A(r_1,r_2))$. This is needed to ensure the existence of the function $\varphi$ claimed in Theorem \ref{Lem:ExistencePhi}. Clearly, in this case condition \eqref{Eq:BCondition} need to be slightly modified.
\end{remark}

\begin{remark}\label{rem-th-noKilling}
Some observations on the statement of Theorem \ref{th-annuli} are in order.\smallskip

 \noindent a) Obviously, if $N$ is compact,  assumption \eqref{radialnorm} is automatically satisfied. In this case, if there exists at least one symmetric solution $u$ of \eqref{dir1}, then each solution must coincide with the symmetric one, thanks to the uniqueness result contained in Theorem \ref{Th:NoneEuclideanDupaigne}. In the opposite direction, the symmetry result could be useful in establishing whether a symmetric solution actually exists. In fact, it is easy to choose a non-linearity $f(t)$ in such a way standard methods to construct a symmetric, say one-dimensional, solution cannot be applied.\smallskip

\noindent b) In the non-compact case, the boundedness assumption \eqref{radialnorm} of Theorem \ref{th-annuli} is apparently weaker then the $W^{1,2}$ global regularity needed in Theorem \ref{Th:NoneEuclideanDupaigne}. Thus, we do not know whether or not there is some global uniqueness of the (stable) solution. \smallskip

\smallskip

\noindent c)  Condition \eqref{Eq:BCondition} is clearly satisfied if $f'(u) \geq -B=0$. As a matter of fact, it will be clear from Lemma \ref{Lem:ExistencePhi} that there is a (strong) stability condition hidden in \eqref{Eq:BCondition}. Indeed, the validity of \eqref{Eq:BCondition}  implies the existence of a smooth solution $\vp >0$ of $\CL \vp \leq 0$ on $\inte M$, where $\CL = \Delta_\Psi - f'(u)$ is the stability operator. According to a classical result independently due to Fischer-Colbrie and Schoen, \cite{FCS}, and to Moss and Piepenbrink, \cite{MP} (see also \cite{De}), we have that $\lambda_{1}^{-\CL}(A(r_{1},r_{2})) \geq 0$. But in fact more is true because we can even obtain that $C^{-1}\leq \vp \leq C$ on the whole $\bar{A}(r_{1},r_{2})$.\smallskip

\noindent d) It would be interesting to note that condition \eqref{Eq:BCondition} can be written as
\begin{align*}
0 \leq \int_{r_1}^{r_2} \frac{\vol_\Psi A(r_1,s)}{\area_\Psi \Sigma_s}\ \ds < \frac{1}{B}
\end{align*}
where the integrand is the inverse of the Cheeger isoperimetric quotient.\smallskip

\item e) From a different perspective, symmetry on Riemannian (warped) products have been previously investigated in \cite{FMV} by A. Farina, L. Mari and E. Valdinoci. Their viewpoint is that of the De Giorgi conjecture where, a-priori, it is not known along which direction the stable solution of the Allen-Cahn type equation is symmetric. Thus, their result takes the form of a geometric splitting of the underlying space. See also \cite{BS} by M. Batista and I.J. Santos for the case of weighted manifolds and negative Ricci lower bounds.
\end{remark}

As a concrete example where to set Theorem \ref{th-annuli} in, we can consider the weighted slabs of Example \ref{Ex:IsoparametricGaussian}, thus obtaining the following

\begin{corollary}\label{Cor:GaussianSlabs}
Let $\bar{A}(r_1,r_2)=[r_1,r_2]\times \mathbb{R}^{n-1} \subset \mathbb{G}^n=\mathbb{R}^n_\Psi$ be a slab in the Gaussian space, whose weight writes as $\Psi(r,\xi)=\frac{r^2}{2}+\frac{|\xi|^2}{2}$.\smallskip
 
Let $u \in C^{4}(\bar {A}(r_{1},r_{2}))$ be a solution of  the Dirichlet problem
\[
\begin{cases}
 \Delta_\Psi u = f (u) & \text{in }A(r_{1},r_{2}) \\
 u \equiv c_{1} & \text{on }\{ r_{1} \}\times N \\
 u \equiv c_{2} &\text{on }\{r_{2}\} \times N.
\end{cases}
\]
where $c_{j} \in \rr$ are given constants and the function $f(t)$ is of class  $C^{2}$ and satisfies $f''(t) \leq 0$. If
\[
\| u \|_{C^{2}_{rad}}< +\infty
\]
and $f'(u)\geq -B$, for some constant $B\geq 0$ satisfying
\begin{align*}
0 \leq B < \left( \int_{r_1}^{r_2} \frac{\int_{r_1}^s e^{-z^2/2}\ \textnormal{d}z}{e^{-s^2/2}}\ \ds \right)^{-1}
\end{align*}
then $u(r,\xi) = \hat u(r)$ is symmetric.
\end{corollary}
\begin{proof}
Thanks to the presence of the Gaussian weight, the leaves of the foliation have finite volume. Thus we can apply Theorem \ref{th-annuli}, obtaining the claim. 
\end{proof}

Observe that this is not true for the same domains in Euclidean space: this fact points out how the presence of a weight that deforms the Riemannian measure may strongly influence the structure of solutions of the equation $\Delta u=f(u)$.

A second important consequence of Theorem \ref{th-annuli} concerns weights with vanishing tangential component.

\begin{corollary}\label{Cor:RadialWeights}
Let $M_\Psi = (I \times_{\s} N)_\Psi$ where $\Psi(r,\xi)=\hat{\Psi}(r)$ is a symmetric smooth function and $(N,g^{N})$ is a complete (possibly non-compact), connected, $(m-1)$-dimensional Riemannian manifold with finite volume $\vol(N)<+\infty$.\smallskip
 
Let $u \in C^{4}(\bar {A}(r_{1},r_{2}))$ be a solution of  the Dirichlet problem
\[
\begin{cases}
 \Delta_\Psi u = f (u) & \text{in }A(r_{1},r_{2}) \\
 u \equiv c_{1} & \text{on }\{ r_{1} \}\times N \\
 u \equiv c_{2} &\text{on }\{r_{2}\} \times N.
\end{cases}
\]
where $c_{j} \in \rr$ are given constants and the function $f(t)$ is of class  $C^{2}$ and satisfies $f''(t) \leq 0$. If
\[
\| u \|_{C^{2}_{rad}} < +\infty
\]
and $f'(u)\geq -B$, for some constant $B\geq 0$ satisfying
\begin{align*}
0 \leq B < \left( \int_{r_1}^{r_2} \frac{\int_{r_1}^s e^{-\Psi(z)}\sigma^{m-1}(z)\ \dz}{e^{-\Psi(s)}\sigma^{m-1}(s)}\ \ds \right)^{-1}
\end{align*}
then $u(r,\xi) = \hat u(r)$ is symmetric.
\end{corollary}


\subsection{Some preliminary lemmas}\label{section-tools-potential} We have already mentioned that the notion of $\CN$-parabolicity, introduced in Section \ref{Section:Parabolicity-NonCaompactMP}, is a kind of compactness from many viewpoints. The following result contains further instances.

\begin{theorem}\label{th-parabolic}
 Let $M_\Psi$ be a weighted Riemannian manifold with (possibly empty) boundary $\partial M$.
 \begin{enumerate}

\item [a)] (Stokes theorem: general vector fields, \cite{IPS}) If $M_\Psi$ is $\CN$-parabolic then, given a vector field $X$ satisfying $|X| \in L^{2}(M, \dv_\Psi)$, $g(X, \vec\nu) \in L^{1}(\partial M,d\textnormal{a}_\Psi)$, $\div_\Psi (X) \in L^{1}(M, \dv_\Psi)$, it holds
\[
\int_{M}\div_\Psi (X)\ \dv_\Psi= \int_{\partial M}g(X,\vec \nu)\ \da_\Psi.
\]
\item [b)](Stokes theorem: gradient vector fields and no boundary, \cite[Prop. 3.1]{GM}) If $M_\Psi$ is parabolic and $\partial M = \emptyset$ then, given $u \in W^{1,2}_{loc}(M, \dv_\Psi)$ satisfying $u \in L^{\infty}(M,\dv_\Psi)$ and $\Delta_\Psi u \in L^{1}(M,\dv_\Psi)$, it holds
\[
\int_{M} \Delta_\Psi u\ \dv_\Psi =0.
\]
\item [c)](Volume growth, \cite{Gr}) Assume that $M_\Psi$ is complete(!) and that $\frac{R}{\vol_\Psi B_{R}(o)} \not\in L^{1}(+\infty)$ for some (any) $o \in \inte M$. Then $M_\Psi$ is $\CN$-parabolic.
\end{enumerate}
\end{theorem}

Keeping the notation and the assumptions of Theorem \ref{th-annuli}, the above potential theoretic tools  enable us to deduce some useful preliminary properties of the $\Psi$-isoparametric domain $\bar A(r_{1},r_{2})$ and of the solution $u$.\smallskip

In view of the next Lemma, recall that $N_\Gamma$ is complete weighted manifold with $\partial N = \emptyset$ and $\vol_\Gamma(N)<+ \infty$.

\begin{lemma}\label{Lem:Parabolicity}
The following hold.
\begin{enumerate}
 \item [i)]  $N_\Gamma$ is parabolic;
 \item [ii)] The closed annulus $\bar A(r_{1},r_{2})_\Psi$ endowed with the weight and the warped product metric inherited from $M_\Psi$ is a weighted $\CN$-parabolic manifold  with  $\partial \bar {A}(r_{1},r_{2}) \not=\emptyset$.
\end{enumerate}
\end{lemma}

\begin{proof}
 i) is a direct consequence of Theorem \ref{th-parabolic}.c. Concerning ii), let $\a = \min_{[r_{1},r_{2}]} \s(r) >0$ and $\b =  \max_{[r_{1},r_{2}]} \s(r) <+\infty$ so that,  on $\bar A(r_{1},r_{2})$,
 \[
 dr \otimes dr + \a \cdot g^{N} \leq g \leq dr \otimes dr + \b  \cdot g^{N}
 \]
in the sense of quadratic forms. Since the LHS metric is complete and the RHS metric has finite $\Psi$-volume the conclusion follows again from  Theorem \ref{th-parabolic}.c.
\end{proof}

For the next Lemma recall also that  $\| u \|_{C^{2}_{rad}} < +\infty$.

\begin{lemma}\label{Lem:IntegralN}
We have
\[
\Delta^{N}_\Gamma u \in L^{\infty}(A(r_{1},r_{2})).
\]
Moreover, for every fixed $\bar r \in [r_{1},r_{2}]$,
\[
\Delta^{N}_\Gamma u\, (\bar r,\cdot) \in L^{1}(N, \textnormal{dv}_\Gamma)
\]
and
\[
\int_{N} \Delta^{N}_\Gamma u  (\bar r,\xi) \, \dv_\Gamma =0.
\]
\end{lemma}

\begin{proof}
 Using the fact that $\Delta_\Psi u = f(u)$ we can write
 \[
 \Delta^{N}_\Gamma u = \s^2 f(u) - \s^2 \partial^{2}_{r}u -\Big((m-1) \s \s'- \Phi'\s^2\Big) \partial_{r}u.
 \]
 From this expression, since $\sup_{[r_{1},r_{2}]}(\s+ |\s'|+|\Phi'|) < +\infty$, $\| u \|_{C^{2}_{rad}} < +\infty$ and, hence, $\sup_{A(r_{1},r_{2})}|f(u)|<+\infty$, we get
 \[
 \Delta^{N}_\Gamma u \in L^{\infty}(A(r_{1},r_{2})).
 \]
 In particular, for every $\bar r \in [r_{1},r_{2}]$,
 \[
 \Delta^{N}_\Gamma u(\bar r, \cdot) \in L^{\infty}(N).
 \]
 Recalling that $\vol_\Gamma(N) < +\infty$ it follows that  $\Delta^{N}_\Gamma u(\bar r, \cdot) \in L^{1}(N,\dv_\Gamma)$. Since  $u(\bar r,\cdot ) \in L^{\infty}(N)$ and $N_\Gamma$ is parabolic without boundary, by Theorem \ref{th-parabolic}.b we conclude that $\int_{N} \Delta^{N}_\Gamma u(\bar r,\xi)\ \dv_{\Gamma}(\xi) = 0$, as required.
\end{proof}

The previous Lemmas, stemming from potential theoretic considerations, will play a fundamental role in the proof of Theorem \ref{th-annuli}. Beside them, we shall also need the validity of the non-compact maximum principle from Proposition \ref{prop:NonCompactPWMaxPrinc}. This follows from the next

\begin{lemma}\label{Lem:ExistencePhi}
There exists a  function $\varphi\in C^2(A(r_{1},r_{2}))\cap C^{0}(\bar{A}(r_{1},r_{2}))$ satisfying condition \eqref{eq:NonCompactPhi} of Proposition \ref{prop:NonCompactPWMaxPrinc}, namely,
\[
\left\{
\begin{array}{ll}
\mathcal{L}\varphi \leq 0 & A(r_{1},r_{2})\\
\frac{1}{C} \leq \varphi \leq C & \bar{A}(r_{1},r_{2}),
\end{array}
\right.
\]
where, as usual, $\CL = \Delta_\Psi - f'(u)$ is the stability operator.
\end{lemma}
\begin{proof}
Let's start by considering the differential inequality $\big(\Delta_\Psi-f'(u)\big)\varphi \leq 0$ when applied to a symmetric function $\varphi(r,\xi)=\varphi(r)$, that is,
\begin{align*}
\varphi'' + \Big((m-1)\frac{\sigma'}{\sigma}-\Phi'\Big)\varphi'-f'(u)\leq 0 \quad  \textnormal{in}\ I=(r_1,r_2).
\end{align*}
Since $f'$ is continuous and $u$ is bounded, then there exists $B\geq 0$ such that
\begin{align*}
-f'(u) \leq B
\end{align*}
Imposing  condition \eqref{Eq:BCondition}, we get  the desired function as the solution of
\begin{align}\label{Eq:Varphi}
\left\{ \begin{array}{ll}
\varphi''+\Big((m-1)\frac{\sigma'}{\sigma}-\Phi'\Big) \varphi'+B=0 & \textnormal{in}\ I\\
\varphi(r_1)=1\\
\varphi'(r_1)=b<0
\end{array} \right.
\end{align}
for a suitable choice of $ b \in \rr$. Indeed, letting
\begin{align*}
& B(t)=B\int_{r_1}^t e^{\Phi(s)}\sigma^{1-m}(s) \int_{r_1}^s e^{-\Phi(z)}\sigma^{m-1}(z)\ \dz\ \ds \geq 0\\
& A(t)=b\ e^{-\Phi(r_1)}\ \sigma^{m-1}(r_1) \int_{r_1}^t e^{\Phi(s)}\sigma^{1-m}(s)\ \ds \leq 0,
\end{align*}
if \eqref{Eq:BCondition} is satisfied, then it is possible to choose $b<0$ such that
\begin{displaymath}
-1<A(r_2)-B(r_2)<0.
\end{displaymath}
It follows that the function
\begin{align*}
\varphi(t)=1+A(t)-B(t)
\end{align*}
is a positive and decreasing solution of \eqref{Eq:Varphi}. In particular, $\varphi$ is bounded above by $\varphi(r_1)=1$, so it clearly solves the differential inequality
\begin{align*}
\varphi''+\Big((m-1)\frac{\sigma'}{\sigma}-\Phi'\Big) \varphi' - f'(u) \varphi \leq \varphi''+\Big((m-1)\frac{\sigma'}{\sigma}-\Phi'\Big) \varphi'+B=0.
\end{align*}
The proof of the Lemma is completed.
\end{proof}

\subsection{Proof of Theorem \ref{th-annuli}}
 Let us define
 \[
 v (r,\xi)= \Delta^{N}_\Gamma u(r,\xi).
 \]
It is enough to show that, for every $\bar r \in [r_{1},r_{2}]$,
\[
\xi \mapsto v(\bar r ,\xi) \text{ is constant on } N.
\]
Indeed, if this is the case, then $u(\bar r, \cdot)$ is a bounded (sub / super) harmonic function on the parabolic weighted manifold $N_\Gamma$, therefore it must be constant on $N$. This is precisely what we have to prove.

Now, since $u$ is (locally) constant on the boundary $\partial A(r_{1},r_{2})$ then
 \[
 v = 0\quad \text{on } \partial A(r_{1},r_{2}).
 \]
On the other hand, using the commutation rule \eqref{Eq:CommutatorLaplacian}, the fact that $\Delta_\Psi u = f(u)$ and the properties of $f$ we see that
\begin{align*}
 \Delta_\Psi v &= \Delta^{N}_\Gamma f(u) \\
 &= \Delta^N f(u) - g^N(\nabla^{N} f(u), \nabla^{N} \Gamma)\\
 &= \div^{N}(\nabla^{N} f(u)) -  f'(u)\ g^N(\nabla^{N} u, \nabla^{N} \Gamma) \\
 &= \div^{N}(f'(u)\nabla^{N}u) - f'(u)\ g^N(\nabla^{N} u, \nabla^{N} \Gamma)\\
 &=f''(u) |\nabla^{N}u|^{2}_{N} + f'(u) \DeltaN u - f'(u)\ g^N(\nabla^{N} u, \nabla^{N} \Gamma)\\
 &\leq f'(u) \DeltaN u - f'(u)\ g^N(\nabla^{N} u, \nabla^{N} \Gamma)\\
 &= f'(u) v.
\end{align*}
Summarizing, the $C^{2}$ function $v$ solves
\[
\begin{cases}
 \Delta_\Psi (-v) \geq f'(u) (-v) & \text{in }A(r_{1},r_{2})\\
 (-v) = 0 & \text{on }\partial A(r_{1},r_{2}).
\end{cases}
\]
By Lemma \ref{Lem:ExistencePhi} we can apply the non-compact Protter-Weinberger maximum principle of Proposition \ref{prop:NonCompactPWMaxPrinc}, and we get
\begin{align*}
v\geq 0\ \textnormal{in}\ A(r_1,r_2).
\end{align*}
On the other hand, 
\begin{align*}
\int_{A(r_{1},r_{2})} v\, \dv_{\Psi} &= \int_{r_{1}}^{r_{2}} \left( \int_{\{t\}\times N} v(t,\xi)\ \dv_\Gamma(\xi) \right)  e^{-\Phi(t)} \s^{m-1}(t) \textnormal{d}t\\
&= \int_{r_{1}}^{r_{2}} \left( \int_{N} \Delta^N_\Gamma u (t,\xi)\ \dv_\Gamma(\xi) \right)  e^{-\Phi(t)} \s^{m-1}(t) \textnormal{d}t\\
&= 0
\end{align*}
where, for the last equality, we have used Lemma \ref{Lem:IntegralN}. As a consequence,
\[
v \equiv 0 \text{ on }A(r_{1},r_{2}),
\]
as required. The proof of the theorem is completed.

\subsection{Infinite annuli} 
Theorem \ref{th-annuli} can be easily generalized to the case of infinite annuli, under suitable assumptions that are trivially satisfied in the case of finite annuli.

To this end, consider $A(r_0, +\infty)=(r_0,+\infty)\times_\s N$ with $r_0 \in \mathbb{R}_{>0}$ and suppose that $\bar{A}(r_0,+\infty)$ is $\CN$-parabolic. If the warping function $\s$ is a bounded function with bounded derivative, then Lemma \ref{Lem:IntegralN} extends trivially to this setting. Moreover, if the function
\begin{align*}
\theta: s\mapsto \frac{\int_{r_0}^s e^{-\Phi(z)}\s^{m-1}(z)\ \dz}{e^{-\Phi(s)} \s^{m-1}(s)}
\end{align*}
is integrable over $(r_0,+\infty)$, then the proof of Lemma \ref{Lem:ExistencePhi} can be readapted, ensuring the existence of the function $\varphi$ and allowing the non-compact Maximum Principle of Theorem \ref{prop:NonCompactPWMaxPrinc} to hold.

In this way, the whole proof of Theorem \ref{th-annuli} can be retraced step by step also in the context of infinite annuli, obtaining the next

\begin{theorem}\label{th-infinite-annuli}
Let $M_\Psi = (\mathbb{R}_{\geq 0} \times_{\s} N)_\Psi$ where $(N,g^{N})$ is a complete (possibly non-compact), connected, $(m-1)$-dimensional Riemannian manifold with finite $\Gamma$-volume $\vol_{\Gamma}(N)<+\infty$ and $\s\in L^\infty(\mathbb{R}_{\geq 0})$ satisfies $\s'\in L^\infty(\mathbb{R}_{\geq 0})$. Suppose also that $\bar {A}(r_{0},+\infty)$ is a $\CN$-parabolic manifold. \smallskip
 
Let $u \in C^{4}(\bar {A}(r_{0},+\infty))$ be a solution of  the Dirichlet problem
\begin{equation}\label{dir1}
\begin{cases}
 \Delta_\Psi u = f (u) & \text{in }A(r_{0},+\infty) \\
 u \equiv c_{0} & \text{on }\{ r_{0} \}\times N.
\end{cases}
\end{equation}
where $c_{0} \in \rr$ is a given constant and the function $f(t)$ is of class  $C^{2}$ and satisfies $f''(t) \leq 0$. If 
\begin{align}
&\| u \|_{C^{2}_{rad}} < +\infty\\
& \Delta^N_\Gamma u \in L^1(N,\textnormal{dv}_\Gamma)\\
&\theta(s)= \frac{\int_{r_0}^s e^{-\Phi(z)}\s^{m-1}(z)\ \dz}{e^{-\Phi(s)} \s^{m-1}(s)} \in L^1(r_0,+\infty)
\end{align}
and $f'(u)\geq -B$, for some constant $B\geq 0$ satisfying
\begin{align}
0 \leq B < \left( \int_{r_0}^{+\infty} \theta(s)\ \ds \right)^{-1}
\end{align}
then $u(r,\xi) = \hat u(r)$ is symmetric.
\end{theorem}

\begin{remark}
Note that, when specified to a model manifold, $A(r_0,+\infty)$ is the exterior domain $\mm(\s) \setminus B_{r}(o)$.
\end{remark}

Theorem \ref{th-infinite-annuli} paves the way for further interesting studies about infinite annuli, such as a deeper understanding of the link between the warping function $\s$ and the weight function $\Psi$. Indeed, it is only in the context of annuli with infinite radius that we can really understand how the behaviour of $\sigma$ at infinity plays a role when combined with that of $\Psi$.

Lastly, it could also be interesting to better understand the $\CN$-parabolicity and its compatibility with the conditions just required for infinite annuli.


\appendix

\section{Proof of the commutation Lemma \ref{lem: Commutation law}}

As we are going to explain, the proof of Lemma \ref{lem: Commutation law} can be obtained by a small variation of the arguments in \cite{savo2018geometric}. First, we need to fix some notation. Given any(!) smooth tube $\O $ with soul $P$, let $U(P)$ be the unit normal bundle of $P$, which is locally isometric to $P\times \mathbb{S}^{n-k-1}$ for $k=\dim(P)$. If we consider the restriction of the exponential map to this bundle
\begin{displaymath}
\begin{split}
\Phi:[-R,R]\times U(P)&\to \O \\
\Big(r,(x, \xi)\Big)=(r,\xi)&\mapsto \textnormal{exp}_x(r \xi)
\end{split}
\end{displaymath}
then  we obtain a diffeomorphism
\begin{displaymath}
\Phi_1:(0,R]\times U(P) \to \O .
\end{displaymath}
In particular, if $(r,\xi)\in [-R,R]\times U(P)$, then there exists a  function $\theta_{\Psi}:[-R,R]\times U(P) \to \mathbb{R}$, positive on $(0,R] \times P$, such that
\begin{displaymath}
\Phi^*(\dv_{\Psi})(r,\xi)=\theta_{\Psi}(r,\xi)\ dr\ d\xi.
\end{displaymath}
If we denote, as usual, $d(x) = \dist(x,P)$, then the function $\theta_{\Psi}$ satisfies
\[
\Delta_{\Psi} d = - \frac{d}{dr} \log(\theta_{\Psi}) = H_{\Psi}.
\]
In particular, if $\O $ is $\Psi$-isoparametric then
\[
\theta_{\Psi}(r,\xi) = \theta_{\Psi}(r)
\]
only depends on the $r$-variable.\smallskip

We are now in the position to give the
\begin{proof}[Proof (of Lemma \ref{lem: Commutation law})]
(a) Let $u \in C^{\infty}(\O )$. Consider the map $F:[0,R]\times U(P)\to \mathbb{R}$ given by the composition $u\circ \Phi$. It extends smoothly to $[-R,R]\times U(P)$ since $\Phi(-r,-\xi)=\Phi(r,\xi)$. Then, if $r>0$ we have
\begin{displaymath}
\int_{\Sigma_r}{u\ \da_{\Psi}} = \int_{U(P)}{F(r,\xi)\ \theta_\Psi(r,\xi)\ \textnormal{d}\xi}.
\end{displaymath}
Since the domain is weighted isoparametric, $\theta_{\Psi}$ depends only on $r$ and hence
\begin{equation}\label{eq:RadialUNormalCoordinates}
\int_{\Sigma_r}{u\ \da_{\Psi}}=\theta_{\Psi}(r)\int_{U(P)}{F(r,\xi)\ \textnormal{d}\xi}.
\end{equation}
Note that, by applying \eqref{eq:RadialUNormalCoordinates} to the constant function $1$, we have
\begin{displaymath}
\area_{\Psi} \Sigma_r:= \int_{\Sigma_{r}} \da_{\Psi} =\theta_{\Psi}(r)\ \area U(P)
\end{displaymath}
where $\vol(U(p))$ denotes the Riemannian measure of $U(P)$. Whence, we can rewrite the averaged function
\[
\hat{u}(r) = \CA_{\Psi}(u)(r)
\]
in the form
\begin{displaymath}
\hat{u}(r)=\frac{1}{\area U(P)}\int_{U(P)}{F(r,\xi)\ \textnormal{d}\xi}.
\end{displaymath}
The proof of statement (a) now follows exactly as in \cite[ Proposition 13]{savo2018geometric}.\smallskip

\noindent (b) Let's start by considering  formula (13) in \cite{savo2018geometric}. It states that
\begin{displaymath}
\frac{d}{dr} \int_{\Lambda_r}{u\ \da}=\int_{\Lambda_r}{\big[ g(\nabla u, \nabla d)+u\ \Delta d \big]\ \da}
\end{displaymath}
where $\Lambda_r$ are the leaves of a smooth tube around $P$. In our context the previous formula becomes
\begin{equation}\label{eq:DerivativeIntegralU}
\begin{split}
\frac{d}{dr}\int_{d=r}{u\ \da_{\Psi}}&=\frac{d}{dr}\int_{d=r}{u\ e^{-\Psi}\ \da}\\
&=\int_{d=r} {\big[ g(\nabla(u\ e^{-\Psi}),\nabla d)+u\ e^{-\Psi}\Delta d \big]\ \da}\\
&=\int_{d=r}{g(\nabla u, \nabla d)\ e^{-\Psi}\ \da}+\int_{d=r}{\big[u\ \Delta d - u\ g(\nabla \Psi,\nabla d)\big]e^{-\Psi}\ \da}\\
&=\int_{d=r}{g(\nabla u, \nabla d)\ \da_{\Psi}}+\int_{d=r}{u\ \Delta_{\Psi} d\ \da_{\Psi}}
\end{split}
\end{equation}
Thanks to the weighted Green identity:
\begin{displaymath}
\begin{split}
\int_{d=r}{g(\nabla u, \nabla d) v\ \da_{\Psi}}
&=\int_{d<r}{g(\nabla u, \nabla v)\ \dv_{\Psi}}+\int_{d<r}{\Delta_{\Psi} u\ v\ \dv_{\Psi}},
\end{split}
\end{displaymath}
and using also that $\Delta_{\Psi} d=H_{\Psi}$, formula \eqref{eq:DerivativeIntegralU} becomes
\begin{displaymath}
\frac{d}{dr}\int_{d=r}{u\ \da_{\Psi}}=\int_{d<r}{\Delta_{\Psi} u\ \dv_{\Psi}}+H_{\Psi} \int_{d =r}{u\ \da_{\Psi}}.
\end{displaymath}
Let
\begin{displaymath}
\begin{split}
& \psi=\int_{d=r}{u\ \da_{\Psi}}\\
& V=\area_{\Psi}\Sigma_r :=\int_{d=r}{\da_{\Psi}}.
\end{split}
\end{displaymath}
Then
\begin{displaymath}
\begin{split}
&\psi'=\int_{d<r}{\Delta_{\Psi} u\ \dv_{\Psi}}+H_{\Psi} \psi\\
&V'=H_{\Psi} V
\end{split}
\end{displaymath}
and the proof of property (b) can  be obtained exactly as in \cite{savo2018geometric}. Indeed, by the fact that $\hat{u}=\frac{\psi}{V}$ we have
\begin{displaymath}
\hat{u}'=\frac{1}{V}\int_{d<r}{\Delta_{\Psi} u\ \dv_{\Psi}}
\end{displaymath}
and thus
\begin{displaymath}
\hat{u}''=-\frac{H_{\Psi}}{V}\int_{d<r}{\Delta_{\Psi}u\ \dv_{\Psi}}+\frac{1}{V}\int_{d=r}{\Delta_{\Psi} u\ \da_{\Psi}}.
\end{displaymath}
This latter, in turn, implies
\begin{displaymath}
\hat{u}''+H_{\Psi} \hat{u}'=\widehat{\Delta_{\Psi} u}.
\end{displaymath}
On the other hand
\begin{displaymath}
\begin{split}
\Delta_{\Psi} (\hat{u}\circ d)&=\Delta(\hat{u}\circ d)-g(\nabla(\hat{u}\circ d),\nabla \Psi)\\
&=\hat{u}''\circ d+(\hat{u}'\circ d) \Delta_{\Psi} d\\
&=\hat{u}''\circ d+(\hat{u}'\circ d) H_{\Psi}\\
&=(\hat{u}''+H_{\Psi} \hat{u}')\circ d
\end{split}
\end{displaymath}
obtaining that
\begin{displaymath}
\widehat{\Delta_{\Psi} u}\circ d=(\hat{u}''+H_{\Psi} \hat{u}')\circ d=\Delta_{\Psi} (\hat{u}\circ d).
\end{displaymath}
This means precisely that
\begin{displaymath}
[\CA_{\Psi},\Delta_{\Psi}]u=0,
\end{displaymath}
as desired.
\end{proof}


\end{document}